%% file: art_new_introduction.tex
\documentclass{article}
\usepackage[utf8]{inputenc}
\usepackage[T2A]{fontenc}
\usepackage[russian,english]{babel}
\usepackage{fullpage} 
\usepackage{parskip} 
\usepackage{tikz} 
\usepackage{amsmath}
\usepackage{hyperref}
\usepackage{amsfonts}
\usepackage{fancyhdr}
\usepackage{changepage}
\usepackage{amssymb}
\usepackage{amsthm}
\usepackage{graphicx}
\usepackage{subcaption}
\usepackage{wrapfig}

\usepackage{mathrsfs}

\usepackage{algorithm}
\usepackage{algpseudocode}

\usepackage{listings}
\usepackage{soul}

\newtheorem{theorem}{Theorem}

\newtheorem{lemma}[theorem]{Lemma}

\newcommand{\io}{\ensuremath{\overline{i}}}
\newcommand{\jo}{\ensuremath{\overline{j}}}
\newcommand{\uo}{\ensuremath{\overline{u}}}
\newcommand{\fo}{\ensuremath{\overline{f}}}

\newcommand{\iz}{\ensuremath{\overline{i}_z}}
\newcommand{\jz}{\ensuremath{\overline{j}_z}}

\newcommand{\num}{\ensuremath{\mathcal{N}}}
\newcommand{\spartial}{\ensuremath{\strut\partial}}

\newcommand{\PL}{\ensuremath{\Phi}}

\newcommand{\jac}{\ensuremath{J_{\xi\eta}}}
\newcommand{\JIT}{\ensuremath{\left(\jac^{T}\right)^{-1}}}

\newcommand{\authorstyle}[1]{{\large\bfseries#1}} 

\newcommand{\institution}[1]{{\footnotesize\itshape#1}} 

\usepackage{titling} 

\newcommand{\HorRule}{\rule{\linewidth}{1pt}} 

\pretitle{
	\vspace{-18pt} 
	\HorRule\vspace{10pt} 
	\fontsize{18}{16}\bfseries\selectfont 
}

\posttitle{\par\vskip 15pt} 

\preauthor{} 

\postauthor{ 
	\vspace{10pt} 
	\par\HorRule 
	\vspace{20pt} 
}

\title{QTT-isogeometric solver in two dimensions} 

\author{
	\authorstyle{L. Markeeva\textsuperscript{1,2}, I. Tsybulin\textsuperscript{3}, I. Oseledets\textsuperscript{1,4}} 
	\newline\newline 
	\textsuperscript{1}\institution{Center for Computational and Data-Intensive Science and Engineering, Skolkovo Institute of Science and Technology, Moscow, Russia}\\
	\textsuperscript{2}\institution{Department of Applied Math and Control, Moscow Institute of Physics and Technology, Moscow, Russia}\\
	\textsuperscript{3}\institution{Laboratory of Fluid Dynamics and Seismics, Moscow Institute of Physics and Technology, Moscow, Russia}\\ 
	\textsuperscript{4}\institution{Institute of Numerical Mathematics of Russian Academy of Sciences, Moscow, Russia} 
}

\begin{document}
	\maketitle
	
	
	
	
	\input{intro}
	\input{discretization}
	\input{qtt}

	\input{assembly}
	\input{experiments}

	\bibliographystyle{unsrt}

\end{document}

%% file: intro.tex
\section{Introduction}

The goal of this paper is to develop a numerical algorithm that solves a \emph{two-dimensional elliptic partial differential equation} in a polygonal domain using \emph{tensor methods} and ideas from \emph{isogeometric analysis}~\cite{bommes2013quad}. The algorithm is based on the \emph{Finite Element} (FE) \emph{approximation}~\cite{strang1973analysis} with \emph{Quantized Tensor Train decomposition} (QTT)~\cite{oseledets2011tensor, khor-qtt-2011, osel-2d2d-2010} used for matrix representation and solution approximation. 

Recently, Kazeev and Schwab~\cite{kazeev2015quantized} have proven that the QTT representation constructed on uniform tensor-product meshes for FE approximations converge exponentially in terms of the effective number $N$ of degrees of freedom. They used the discontinuous Galerkin method with QTT on a curvilinear polygon domain. QTT-format is used as a low-rank approximation for tensors. One of the ideas that were introduced in~\cite{kazeev2015quantized} is a special ordering of elements called \emph{transposed QTT} (identical to z-order curve~\cite{morton1966computer}) which decreases TT-ranks. 
\begin{figure}[h]
	\centering
	\includegraphics[width=0.65\textwidth]{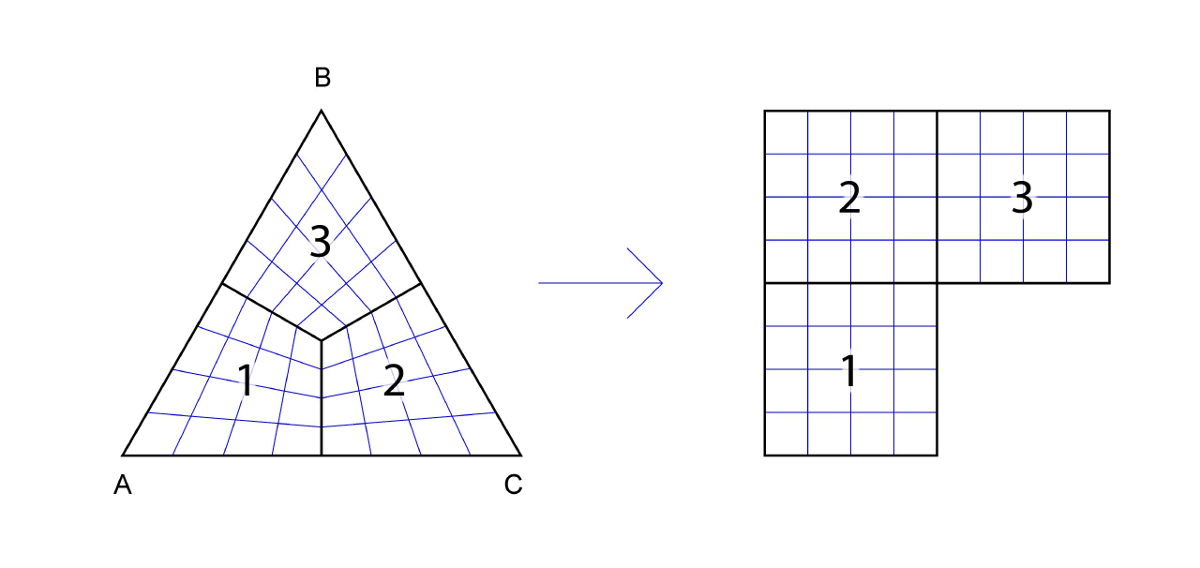}
	\caption{A triangle decomposition into three quadrangles.}
	\label{img:triangle}
\end{figure}

The main problem which was not addressed in~\cite{kazeev2015quantized} is how to construct a global stiffness matrix in the QTT-format with low approximation ranks ``on the fly''. This is important because it leads to a lower memory consumption. Additionally, the idea in~\cite{kazeev2015quantized} was proposed for a different method (discontinuous Galerkin vs \emph{Finite Element Method} (FEM)). The present paper provides a numerical algorithm which uses ideas of isogeometric analysis to decompose a domain into a set of adjacent quadrangular domains, builds a quadrangle mesh on each subdomain, builds a stiffness matrix for the solution on each of the meshes, and concatenates these matrices between adjacent subdomains. The global stiffness matrix is constructed by putting the stiffness matrices from each subdomain on the main diagonal and special interface matrices to off-diagonal blocks. 
Subdomain stiffness matrices are constructed with \emph{the Neumann boundary conditions} on the inter-domain boundaries and with \emph{the Dirichlet boundary conditions} on the domain boundary. Interface matrices represent simple auxiliary equations that are imposed on every internal boundary between each pair of adjacent subdomains. These equations also turn distinct Neumann boundary conditions for each subdomain into the correct FEM discretization of the original equation on the inter-domain boundaries.
This paper explains main ideas on a simple example which is shown in Figure~\ref{img:triangle}.

If we use the natural ordering (row-by-row) of unknowns during the global stiffness matrix construction, the resulting matrix has QTT-ranks of at least $2^d$, where $2^d$ is the number of nodes per quadrangle side, which makes method inapplicable. In this paper we expand the idea of transposed QTT and introduce the technique for building QTT matrix approximation with \emph{z-ordering} (see Figure~\ref{img:ordering}) with the help of Kronecker product.

The key points of the paper are:
\begin{itemize}
	\item we propose a special discretisation scheme that allows to construct the global stiffness matrix in the QTT-format. This algorithm has $O(\log n)$ complexity, where $n=2^d$ is the number of nodes per quadrangle side;

	\item we propose a new operation in QTT-format called the \emph{z-kron operation}, which makes it possible to build a matrix in z-order if the matrix can be described in terms of Kronecker products and sums;
	
	\item we present an algorithm for building a QTT coefficient matrix in z-order for FEM ``on the fly'' as opposed to the transformation of a calculated matrix into QTT. This algorithm has $O(\log n)$ complexity, where $n=2^d$ is the number of nodes per quadrangle side.
\end{itemize}	

%% file: discretization.tex
\section{Discretization scheme}

\subsection{Model problem}\label{sec:model_problem}

In this paper we consider the following two-dimensional model problem:
\begin{equation} \label{eq:Dirichlet_problem}
\left\{
\begin{array}{l}
-\Delta u = f\mbox{ in }\Omega, \\
\left.u\right|_{\partial \Omega} = 0,
\end{array}
\right.
\end{equation}

here $\Omega\subseteq \mathbb{R}^2$, $\Omega$ is a polygon in the two-dimensional space, with the boundary $\partial\Omega$. We divide $\Omega$ into a set of $q$ \emph{adjacent quadrangles} (\emph{subdomains}):

\begin{equation}
\Omega=\bigcup_{m=1}^q\Omega_m,
\end{equation}

where each $\Omega_m$ is a quadrangle. We map each quadrangle $\Omega_m$ into the standard square $\mathbb{K} = [-1, 1]^2$ using bilinear transformation. In each quadrangle $\Omega_m$ the FEM basis is introduced (details are in Sections \ref{sec:fem_formulation}, \ref{sec:jacobian_linearity}). We build a $2^d\times 2^d$ tensor-product mesh on each $\Omega_m$. 
The number of degrees of freedom in each mesh is equal to the number of nodes in each mesh. It is very important that unknowns that correspond to nodes on the inner boundaries are introduced for both adjacent quadrangles. This fact allows us to have $4^{d}$ unknowns and additional ``consistency'' equations need to be introduced, see Section \ref{sec:solution_concatenation}. It can be shown that the final linear system will have the following form: 
\begin{equation}\label{eq:fem_concatinated}
Bu=g\text{, where }
B = \begin{pmatrix}
B_{11} & B_{12} & \dotsc & B_{1q}\\
B_{21} & B_{22} & \dotsc & B_{2q}\\
\vdots & \vdots & \ddots & \vdots\\
B_{q1} & B_{q2} & \dotsc & B_{qq}
\end{pmatrix},\ 		
u = \begin{pmatrix}
u^{\left(1\right)}\\			
\vdots\\
u^{\left(q\right)}\\
\end{pmatrix},\ 
g=\begin{pmatrix}
g^{\left(1\right)}\\
\vdots\\
g^{\left(q\right)}
\end{pmatrix},
\end{equation}

\begin{equation}\label{eq:stiffness_interfaces}
B_{mm} = A^{\left(m\right)}-\gamma \Pi_{mm},\ 
B_{mp}=\begin{cases}
\Pi_{mp}A^{\left(p\right)}-\gamma\Pi_{mp},&\text{ if $m\neq p$ and $m$ and $p$ are adjacent},\\
0,&\text{ otherwise,}
\end{cases}
\end{equation}

\begin{equation}\label{eq:force_interfaces}
g^{\left(m\right)}=f^{\left(m\right)}+\sum_{p\neq m}\Pi_{mp}f^{\left(p\right)},
\end{equation}

where $B$ is a \emph{global stiffness matrix}, $g$ is a \emph{global stiffness vector}, $u$ is a vector of unknowns for the whole system. $m$ and $p$ are indices of some quadrangles, $A^{(m)}$ is a \emph{stiffness matrix} for $\Omega_m$, $u^{(m)}$ are vectors of unknowns for $\Omega_m$, $f^{(m)}$ is a \emph{force vector} for $\Omega_m$. Note that, $B_{mp}$ is not equal to zero if and only if $m$-th and $p$-th quadrangles are adjacent. $\Pi_{mm}$ and $\Pi_{mp}$ are special permutation matrices, we describe them in detail in Section \ref{sc:concatination_matrix}. $\gamma$ is an additional multiplier which is approximately equal to the corresponding diagonal elements of in $A^{\left(m\right)}$, see Section \ref{sec:solution_concatenation} for details.

The construction of the matrix $B$ and the vector $g$ in z-order describes by Algorithm~\ref{alg:final_alg}\footnote{Implementation is available at https://github.com/RerRayne/qtt-laplace}.

\begin{algorithm}
	\caption{Build the stiffness matrix and the force vector}
	\label{alg:final_alg}
	\begin{algorithmic}[1]
		\Require Divide a polygonal field $\Omega$ into a set of adjacent quadrangles $\Omega=\cup_{m=1}^q\Omega_m$
		\State Generate z-order mesh grid matrices $\iz$ and $\jz$, as described in Section \ref{sec:z-order-mesh-grid}
		\For {a mesh on each $\Omega_m$}
		\State Build a quadrangle mesh with $4^d$ nodes for each mapped quadrangle
		\State Map all finite elements from the mesh to the square $\mathbb{K}=[-1, 1]^2$ using vectorized operations
		\State In each finite element calculate $J$ and $J^A$ at the center, using Theorem \ref{th:linear_jacobian}.
		\State Init a global stiffness matrix $B^{(m)}$ by zeros
		\State Init a global force vector $g^{(m)}$ by zeros
		\For {all pairs of shape functions $c_1=(i_1, j_1)$, $c_2=(i_2, j_2)$ on finite elements}
		\State Calculate the vector $K_{c_1 c_2}$. See Section \ref{sec:stiffness_global}\label{st:K_l}
		\State Calculate vector $G_{c_1 c_2}$ as described in Section \ref{sec:stiffness_global}
		\State Calculate $A_{c_1 c_2}^{(m)}$ according to the formula \eqref{eq:global_stiffness_matrix}
		\State Calculate $f_{c_1 c_2}^{(m)}$ according to the formula \eqref{eq:global_force_vector}
		\State $B^{(m)} += A_{c_1 c_2}^{(m)}$
		\State $g^{(m)} += f_{c_1 c_2}^{(m)}$\label{st:g_sum}
		\EndFor
		\EndFor
		\State Build concatenation matrices for subdomains as described in Section \ref{sc:concatination_matrix}
		\State Build matrices $B_{ij}$, $i,j=1\ldots q$ by applying formula \eqref{eq:stiffness_interfaces} to $B^{(m)}$
		\State Build vectors $g^{(m)}$ by applying formula \eqref{eq:force_interfaces} to $f^{(m)}$
		\State Put all $B_{ij}$ into the resulting stiffness matrix $B$ by applying boundary conditions from Section \ref{sec:condition_on_subdomain}.
		\State Put all $g^{(m)}$ into the resulting force vector $g$ by applying boundary conditions from Section \ref{sec:condition_on_subdomain}\label{st:build_global_g}\\
		\Return $B$ and $g$
	\end{algorithmic}
\end{algorithm}


\subsection{Construction of a diagonal block}\label{sec:fem_formulation}

Each diagonal block of the matrix $B$ in \eqref{eq:fem_concatinated} is a sum of two matrices. The first matrix is the FEM stiffness matrix on a single quadrangle with the Dirichlet boundary conditions on the original boundary and the Neumann boundary condition on the inner boundaries between two adjacent quadrangless. The second part helps to build additional ``consistency'' equations for the concatenation along the boundary between adjacent quadrangles. We explain the idea of this concatenation in Sections \ref{sec:solution_concatenation}, \ref{sc:concatination_matrix}-\ref{sec:concatination_by_vertex}.

Thus to construct the diagonal block we need to form the stiffness matrix for a boundary value problem on a single quadrangle. The problem is:
\begin{equation} \label{eq:domain_equation}
	\left\{
	\begin{array}{rl}
	-\Delta u = f, &  \text{ in } \Omega_m,\\
	u = 0, & \text{ on domain boundary } \partial \Omega_m \cap \partial \Omega,\\
	\frac{\partial u}{\partial n} = 0, & \text{ on internal boundary } \partial \Omega_m \setminus \partial \Omega.
	\end{array}
	\right.
\end{equation}

Construction of the stiffness matrix in the QTT-format for such problems has been considered in~\cite{dolgov2015simultaneous,kazeev2013low} and we follow a similar approach. Consider the classical FEM formulation on a quadrangle mesh:

\begin{equation}\label{eq:uf-decomp}
u_h(x, y) = \sum_{i=1}^{n}\sum_{j=1}^{n}u_{ij}\phi_{ij}(x, y),\ f_h(x, y) = \sum_{i=1}^{n}\sum_{j=1}^{n}f_{ij}\phi_{ij}(x, y),
\end{equation}

where index $h$ denotes discretized problem, $u_h\left(x, y\right)$ and $f_h\left(x, y\right)$ are expressed in the basis of \emph{shape functions} $\phi_{ij}$. 

Using the standard Galerkin method we get:
\begin{align}\label{eq:FEM}
A\overline{u} = M\overline{f}, \mbox{ where } A_{\num_{ij}\num_{kl}} =\iint_{\Omega_{m}} \left(\nabla \phi_{ij}, \nabla \phi_{kl}\right) dxdy,\ M_{\num_{ij} \num_{kl}}= \iint_{\Omega_{m}} \left(\phi_{ij},\phi_{kl}\right)dxdy,
\end{align}

\begin{equation}
\uo = \begin{bmatrix}
u_{1,1}\\
u_{1,2}\\	
\vdots\\
u_{n,n}\\
\end{bmatrix}, \ \fo = \begin{bmatrix}
f_{1,1}\\
f_{1,2}\\	
\vdots\\
f_{n,n}\\
\end{bmatrix}.
\end{equation}

The matrix $A\in \mathbb{R}^{n^2\times n^2}$ is a \emph{stiffness matrix}, $M\in \mathbb{R}^{n^2\times n^2}$ is a \emph{mass matrix}. $A, M \in \mathbb{R}^{n^2\times n^2}$. Here we introduce a mapping $\num_{ij}$ from index pairs $(i,j)$ into some elements ordering. 

We map each finite element $F_e$ to $\mathbb{K}$ bilinearly: $h:F_e\rightarrow \mathbb{K}$, where $\mathbb{K} = \left[-1, 1\right]^2$. Points from $F_{e}$ use $x, y$ to denote coordinates, while points from $\mathbb{K}$ use $\xi, \eta$ to denote coordinates. 

This leads to the following stiffness and mass matrices:
\begin{align}\label{eq:Aij}
A_{\mathcal N_{ij} \mathcal N_{kl}} &= 
\iint_{\Omega_{m}} \left(\nabla \phi_{ij}, \nabla \phi_{kl}\right) dxdy =\\ 
&=\sum_{\substack{F_e\text{ such that }\\ \left(\nabla \Phi_{\tau_{F_e, ij}}, \nabla \Phi_{\tau_{F_e, kl}}\right) \not \equiv 0}}
\iint_{\mathbb{K}} 
\left(
\JIT \nabla\Phi_{\tau_{F_e, ij}},
\JIT \nabla\Phi_{\tau_{F_e, kl}}
\right)
\left|\jac\right| d\xi d\eta,
\end{align}

\begin{equation}
\label{eq:Mij}
M_{\mathcal N_{ij} \mathcal N_{kl}} = \iint_{\Omega_{m}} \left(\phi_{ij},\phi_{kl}\right)dxdy = \sum_{\substack{F_e\text{ such that }\\ \left(\Phi_{\tau_{F_e, ij}}, \Phi_{\tau_{F_e, kl}}\right) \not \equiv 0}} \iint_{\mathbb{K}} \left(\Phi_{\tau_{F_e, ij}}, \Phi_{\tau_{F_e, kl}}\right)\left|\jac\right|d\xi d\eta,
\end{equation}                                                                                          

where $\tau_{F_e, ij}$ is a corner point of $\mathbb{K}$ which corresponds to the corner $(i, j)$ of $F_e$. We will omit the index $F_e$ of $\tau_{F_e, ij}$ where not ambiguous.
$\Phi_{\tau_{ij}}$ is the shape function at $\left(\xi, \eta\right)$ corresponding to $\phi_{ij}$ on a mapped finite element $F_e$ to $\mathbb{K}$, 
$\jac = 
\begin{bmatrix}
\dfrac{\spartial x}{\spartial \xi}& \dfrac{\spartial x}{\spartial \eta}\\
\dfrac{\spartial y}{\spartial \xi}& \dfrac{\spartial y}{\spartial \eta}
\end{bmatrix}$ 
at point $(\xi,\eta)$. Here we assume that $\Phi_{\tau_{ij}}$ is one of the \emph{Lagrange shape functions}, which are described in Section \ref{sec:jacobian_linearity}.

The relation between $\nabla \Phi_{\tau_{ij}}$ and $\nabla \phi_{ij}$ is:
\begin{equation}
\begin{bmatrix}
\dfrac{\spartial}{\spartial x} \\
\dfrac{\spartial}{\spartial y}
\end{bmatrix}
\phi_{ij} 
= 
\begin{bmatrix}
\dfrac{\spartial}{\spartial \xi}  \dfrac{\spartial \xi} {\spartial x} + \dfrac{\spartial}{\spartial \eta} \dfrac{\spartial \eta}{\spartial x} \\
\dfrac{\spartial}{\spartial \xi}  \dfrac{\spartial \xi} {\spartial y} + \dfrac{\spartial}{\spartial \eta} \dfrac{\spartial \eta}{\spartial y} 
\end{bmatrix} 
\Phi_{\tau_{ij}} 
= 
\begin{bmatrix}
\dfrac{\spartial \xi}{\spartial x} & \dfrac{\spartial \eta}{\spartial x}\\
\dfrac{\spartial \xi}{\spartial y} & \dfrac{\spartial \eta}{\spartial y}
\end{bmatrix}
\begin{bmatrix}
\dfrac{\spartial}{\spartial \xi}\\
\dfrac{\spartial}{\spartial \eta}
\end{bmatrix}
\Phi_{\tau_{ij}}
=
\JIT \nabla\Phi_{\tau_{ij}}.
\end{equation}

\subsection{Solutions concatenation}\label{sec:solution_concatenation}

As was described in Section~\ref{sec:model_problem}, $\Omega$ is divided into $q$ adjacent quadrangles $\Omega=\cup_{m=1}^q\Omega_m$. In each subdomain a mesh with $n\times n$ nodes is built. External boundaries have the Dirichlet boundary condition, boundaries between adjacent quadrangles have the Neumann boundary condition. For each $\Omega_m$ a coefficient matrix $A^{(m)}$ and a vector $f^{(m)}$ are built (see Section~\ref{sec:fem_formulation}). After matrices are built, we need be sure that some quadrangles should ``concatenate'' their solution along with the inner boundary. This section explains the way how to describe this ``concatenation'' with the help of the matrix $A^{(m)}$ and the vector $f^{(m)}$. 

We can rewrite \eqref{eq:FEM}, adding corresponding markers $(m)$ as:
\begin{equation}\label{eq:fem_approx_at_node}
\sum_{k,l=1}^{n} A_{\num_{ij}\num_{kl}}^{(m)}\uo_{\num_{kl}}^{(m)} = 
f_{\num_{ij}}^{(m)},
\end{equation}

where $i,j\in[1\ldots n]$ specify a position of the node in the mesh, $m$ is the number of the quadrangle in which we are solving the problem.

Assume that during the concatenation of solutions in subdomains $m_1, m_2, \dots, m_r$ nodes $(i_1, j_1), \dotsc (i_r, j_r)$, where $r \geq 2$, are merged. Our goal is to get the right solution approximation at the joined nodes. For that, we have to replace Equation~\eqref{eq:fem_approx_at_node} at each node $i_k$ with:
\begin{equation}\label{eq:concatenation_replacement}
\sum_{k,l=1}^{n} A_{\num_{i_1 j_1}\num_{kl}}^{(m_1)}\uo_{\num_{kl}}^{(m_1)} +
\ldots +
\sum_{k,l=1}^{n} A_{\num_{i_r j_r}\num_{kl}}^{(m_r)}\uo_{\num_{kl}}^{(m_r)}
= f_{\num_{ij}}^{(m_1)} + \ldots + f_{\num_{ij}}^{(m_r)}.
\end{equation}

Additionally, we need $r-1$ equations which describe continuity at the joined nodes:
\begin{equation}\label{eq:continuty_requirements}
\uo^{(m_1)}_{\num_{i_1 j_1}}=\uo^{(m_2)}_{\num_{i_2 j_2}}, \uo^{(m_1)}_{\num_{i_1 j_1}}=\uo^{(m_3)}_{\num_{i_3 j_3}}, \dotsc, \uo^{(m_1)}_{\num_{i_1 j_1}}=\uo^{(m_r)}_{\num_{i_r j_r}}.
\end{equation}

However, \eqref{eq:continuty_requirements} does not follow the structure of \eqref{eq:FEM} and decreases the diagonal dominance of $A$, hence it should be modified.

By taking various linear combinations we can reformulate \eqref{eq:continuty_requirements} in the following symmetric way:
\begin{equation}
\left\{
\begin{array}{c@{}c@{\,}c@{}c@{\;}c@{\;}c@{\,}c@{}c@{\;}c}
\displaystyle (r-1) \uo^{(m_1)}_{\num_{i_1 j_1}} &-& \uo^{(m_2)}_{\num_{i_2 j_2}} &-& \dots &-& 
\uo^{(m_r)}_{\num_{i_r j_r}} &=& 0,\\
\displaystyle - \uo^{(m_1)}_{\num_{i_1 j_1}} &+& (r-1)\uo^{(m_2)}_{\num_{i_2 j_2}} &-& \dots &-& 
\uo^{(m_r)}_{\num_{i_r j_r}} &=& 0,\\
&&&&\ddots\\
\displaystyle -\uo^{(m_1)}_{\num_{i_1 j_1}} &-& \uo^{(m_2)}_{\num_{i_2 j_2}} &-& \dots &+& 
(r-1)\uo^{(m_r)}_{\num_{i_r j_r}} &=& 0.
\end{array}
\right.
\end{equation}
This is a system with $r\times r$ matrix of rank $r-1$. To make it nonsingular let us add \eqref{eq:concatenation_replacement} to each equation. For the sake of readability consider $r = 2$ and $\mathcal I_t = \mathcal N_{i_tj_t}, \mathcal J = \mathcal N_{kl}$:
\begin{equation}\label{eq:equality_reformulation}
\left\{
\begin{array}{c}
\displaystyle 
\phantom{-}\gamma\uo^{(m_1)}_{\mathcal I_1} + 
\sum_{\mathcal J=1}^{n^2} A_{\mathcal I_1 \mathcal J}^{(m_1)}
\uo_{\mathcal J}^{(m_1)} 
-
\gamma\uo^{(m_2)}_{\mathcal I_2} + 
\sum_{\mathcal J=1}^{n^2} A_{\mathcal I_2 \mathcal J}^{(m_2)}
\uo_{\mathcal J}^{(m_2)}
= 
f_{\mathcal I_1}^{(m_1)}
+ 
f_{\mathcal I_2}^{(m_2)},
\\
\displaystyle 
- \gamma\uo^{(m_1)}_{\mathcal I_1} +
\sum_{\mathcal J=1}^{n^2} A_{\mathcal I_1\mathcal J}^{(m_1)}
\uo_{\mathcal J}^{(m_1)} 
+
\gamma\uo^{(m_2)}_{\mathcal I_2} + 
\sum_{\mathcal J=1}^{n^2} A_{\mathcal I_2\mathcal J}^{(m_2)}
\uo_{\mathcal J}^{(m_2)}
=
f_{\mathcal I_1}^{(m_1)}
+ 
f_{\mathcal I_2}^{(m_2)}
,
\end{array}
\right.
\end{equation}
here $\gamma$ is an additional multiplier with the magnitude and sign close to that of diagonal elements of $A^{(m_k)}$. 

To rewrite this operation in the matrix form let us now introduce a matrix $\Pi$ that maps between indices $\mathcal I_1$ and $\mathcal I_2$, that is
\[
\Pi_{\mathcal I_1 \mathcal I_2} = \begin{cases}
1, &\mbox{if node } \mathcal I_1 \mbox { from } \Omega^{(m_1)} \mbox{ should be concatenated with node } \mathcal I_2 \mbox{ from } \Omega^{(m_2)},\\
0, &\mbox{otherwise}.
\end{cases} 
\]
The whole FEM system can now be expressed 
\begin{equation}
\begin{pmatrix}
\gamma \Pi \Pi^\top + A^{(m_1)} 
&
-\gamma \Pi + \Pi A^{(m_2)}\\
-\gamma \Pi^\top + \Pi^\top A^{(m_1)} 
&
\gamma \Pi^\top \Pi + A^{(m_2)}\\
\end{pmatrix}
\begin{pmatrix}
\uo^{(m_1)}\\
\uo^{(m_2)}
\end{pmatrix} = 
\begin{pmatrix}
I & \Pi\\
\Pi^\top & I\\
\end{pmatrix}
\begin{pmatrix}
f^{(m_1)}\\
f^{(m_2)}
\end{pmatrix}.
\end{equation}
This operation has to be applied to each pair of adjacent subdomains.

\subsection{Structure of Jacobian}\label{sec:jacobian_linearity}

\begin{figure}[ht!]
\centering
\includegraphics[width=.5\textwidth]{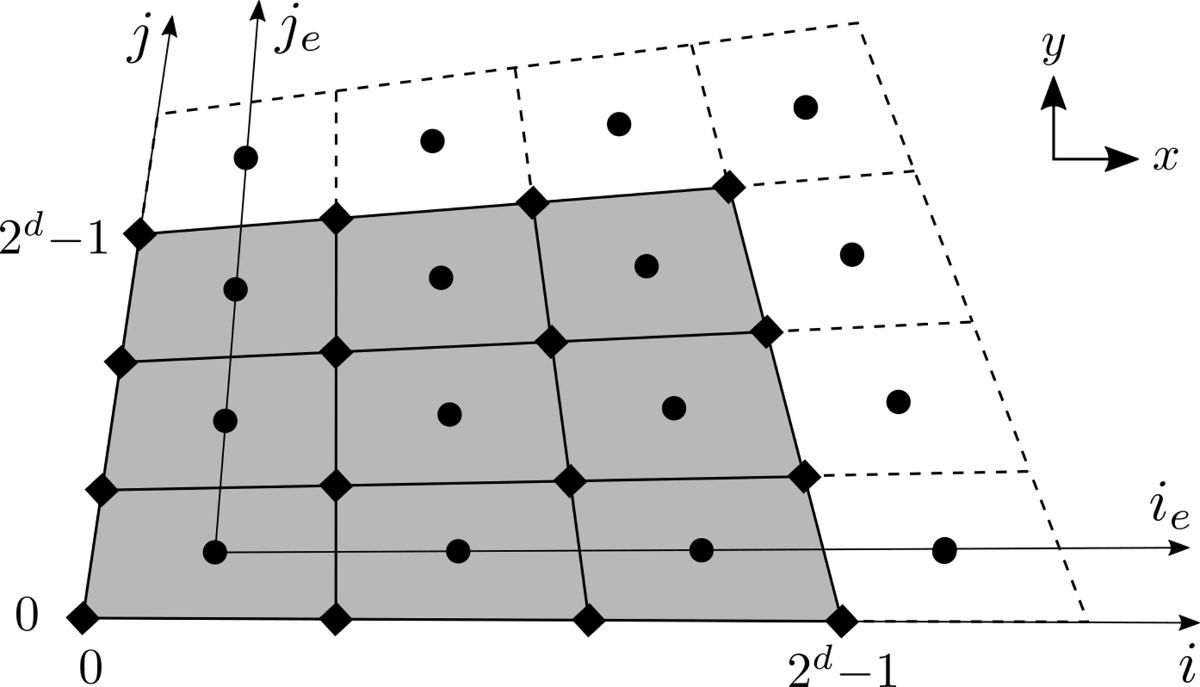}
\caption{The mesh in $\Omega_m$ domain. Here indices $i, j$ enumerate mesh vertex and 
$i_e, j_e$ enumerate mesh elements. Padding elements are dashed.}
\label{fig:omega_m}
\end{figure}

Recall that each quadrangle $\Omega_m$ is mapped onto the unit square $\mathbb K$. $\mathbb K$ is uniformly split into $(2^d-1) \times (2^d-1)$ elements and each is mapped back to the original quadrangle. Figure~\ref{fig:omega_m} presents the results of this procedure in the mesh.

Since QTT normally deals with indices that have $2^d$ range, not $2^d-1$, a single layer of padding elements is added to the mesh. These elements are not used in any way and exist only for a padding purpose. 

The coordinates of vertex -- $x_{i,j}, y_{i,j}$ are bilinear functions of $i, j$ (by construction). To give explicit expressions for them let us first introduce the following shape functions (\emph{Lagrange basis} on the $\mathbb K$):
\begin{equation}\label{eq:phis}
\begin{aligned}
\PL_{-1, -1}\left(\xi, \eta\right) &= \dfrac{\left(1-\xi\right)\left(1-\eta\right)}{4},\qquad
\PL_{1, -1}\left(\xi, \eta\right) = \dfrac{\left(1+\xi\right)\left(1-\eta\right)}{4},\\
\PL_{1, 1}\left(\xi, \eta\right) &= \dfrac{\left(1+\xi\right)\left(1+\eta\right)}{4},\qquad
\PL_{-1, 1}\left(\xi, \eta\right) = \dfrac{\left(1-\xi\right)\left(1+\eta\right)}{4},
\end{aligned}
\end{equation}  
where indices of $\PL$ denote the corner point of $\mathbb{K}$ where corresponding $\PL$ is equal to one. This kind of shape functions entails the following property of a Jacobian that is important for this paper.

Given a quadrangle $Q$ with vertices $(x_1, y_1)$, $(x_2, y_2)$, $(x_3, y_3)$, $(x_4, y_4)$ it is easy now to give the mapping $r$ from the standard square $\mathbb K$ to $Q$:
\[
	r(\xi, \eta) = 
	\begin{bmatrix}
	x\left(\xi, \eta\right)\\
	y\left(\xi, \eta\right)\\
	\end{bmatrix} = 
	\begin{bmatrix}
	x_1 & x_2 & x_3 & x_4\\
	y_1 & y_2 & y_3 & y_4
	\end{bmatrix}
	\begin{bmatrix}
	\PL_{-1, -1}(\xi, \eta)\\
	\PL_{1, -1}(\xi, \eta)\\
	\PL_{1, 1}(\xi, \eta)\\
	\PL_{-1, 1}(\xi, \eta)
	\end{bmatrix}.
\]
Obviously the map is bilinear in $\xi, \eta$ and each corner of $\mathbb K$ is mapped to the corresponding corner of $Q$. 

Taking $\Omega_m$ as $Q$ we can now express the coordinates of the grid vertices:
\[
\begin{bmatrix}
x_{i,j}\\
y_{i,j}
\end{bmatrix} = 
r\left(
2\frac{i}{2^d-1}-1, 2\frac{j}{2^d-1}-1
\right) \equiv \tilde r(i, j).
\]
It is clear that $\tilde r$ is a bilinear function of $i, j$ so it can be written in form $\tilde r(i, j) =  q_0 + q_x i + q_y j + q_{xy} ij$, where $q_0, q_x, q_y, q_{xy}$ are some vectors.

Consider a finite element with index $(i, j)$. Its vertices have indices $(i, j), (i+1, j), (i, j+1)$ and $(i+1, j+1)$. One may easily verify that the element's center is located at
\[
\begin{bmatrix}
\dfrac{x_{i,j} + x_{i+1,j} + x_{i,j+1} + x_{i+1, j+1}}{4}\\[9pt]
\dfrac{y_{i,j} + y_{i+1,j} + y_{i,j+1} + y_{i+1, j+1}}{4}\\
\end{bmatrix}
= r\left(i + \frac{1}{2}, j + \frac{1}{2}\right).
\]
Thus the mapping $\tilde r$ gives not only the vertices' coordinates, but also the coordinates of the elements' centers.

To construct the FEM matrices we need to map each finite element to $\mathbb K$ and evaluate the Jacobian matrix of the mapping for the each finite element. This operation is greatly simplified by the following Lemma:
\newcommand{\pd}[2]{\frac{\partial #1}{\partial #2}}
\newcommand{\adj}{\operatorname{adj}}
\begin{lemma}\label{th:linear_jacobian}
	Consider a finite element $Q_{ij}$ with vertices at 
	$(x_{i,j}, y_{i,j})$,
	$(x_{i+1,j}, y_{i+1,j})$,
	$(x_{i,j+1}, y_{i,j+1})$,
	$(x_{i+1,j+1}, y_{i+1,j+1})$. Let $r_{ij}(\xi, \eta)$ be the mapping from $\mathbb K$ to $Q_{ij}$. Then the Jacobian matrix of this mapping $J^{(i,j)}(\xi, \eta) = \pd{r_{ij}(\xi, \eta)}{(\xi, \eta)}$ is a linear function in $i,j$, that is:
	$$
	J^{(i,j)} (\xi, \eta) = J^{(0,0)} (\xi, \eta) 
	+ i J^{(1,0)} (\xi, \eta)
	+ j J^{(0,1)} (\xi, \eta),
	$$
	and also it its determinant:
	$$
	\left| J^{(i,j)}(\xi, \eta) \right| = \left| J^{(0,0)} (\xi, \eta) \right| 
	+ i \left| J^{(1,0)} (\xi, \eta)\right|
	+ j \left| J^{(0,1)} (\xi, \eta)\right| .
	$$
\end{lemma}

\begin{proof}
	The proof is straightforward and is done by direct evaluation of $J^{(i,j)}(\xi, \eta)$.
	First, recall that the mapping $r_{ij}$ may be expressed as:
	\[
	r_{ij}(\xi, \eta) = 
	\begin{bmatrix}
	x_{i,j} & x_{i+1,j} & x_{i+1,j+1} & x_{i,j+1}\\
	y_{i,j} & y_{i+1,j} & y_{i+1,j+1} & y_{i,j+1}
	\end{bmatrix}
	\begin{bmatrix}
	\PL_{-1, -1}(\xi, \eta)\\
	\PL_{1, -1}(\xi, \eta)\\
	\PL_{1, 1}(\xi, \eta)\\
	\PL_{-1, 1}(\xi, \eta)
	\end{bmatrix}.
	\]
	Since $\PL_{-1, -1}(\xi, \eta) + \PL_{1, -1}(\xi, \eta) + \PL_{1, 1}(\xi, \eta) + \PL_{-1, 1}(\xi, \eta) \equiv 1$ we can rewrite $r_{i,j}$ in an equivalent form:
	\[
	r_{ij}(\xi, \eta) = 
	\begin{bmatrix}
	x_{i,j}\\
	y_{i,j}
	\end{bmatrix}
	+	
	\begin{bmatrix}
	0 & x_{i+1,j} - x_{i,j} & x_{i+1,j+1} - x_{i,j} & x_{i,j+1} - x_{i,j}\\
	0 & y_{i+1,j} - x_{i,j} & y_{i+1,j+1} - x_{i,j} & y_{i,j+1} - x_{i,j}
	\end{bmatrix}
	\begin{bmatrix}
	\PL_{-1, -1}(\xi, \eta)\\
	\PL_{1, -1}(\xi, \eta)\\
	\PL_{1, 1}(\xi, \eta)\\
	\PL_{-1, 1}(\xi, \eta)
	\end{bmatrix}.
	\]
	Let us plug $x_{i,j} = \tilde r(i, j) = q_0 + q_x i + q_y j + q_{xy} ij$. The coefficients $q_0,q_x,q_y,q_{xy}$ are the same vectors for the every finite element in the $\Omega_m$ domain. Finally we obtain
	\[
	r_{ij}(\xi, \eta) = 
	\begin{bmatrix}
	x_{i,j}\\
	y_{i,j}
	\end{bmatrix}
	+	
	\begin{bmatrix}
	0 & q_x + q_{xy}j & q_x + q_y + q_{xy}(i+j+1) & q_y + q_{xy} i
	\end{bmatrix}
	\begin{bmatrix}
	\PL_{-1, -1}(\xi, \eta)\\
	\PL_{1, -1}(\xi, \eta)\\
	\PL_{1, 1}(\xi, \eta)\\
	\PL_{-1, 1}(\xi, \eta)
	\end{bmatrix}.
	\]
	By differentiating with respect to $\xi, \eta$ we obtain:
	\[
	J^{(i,j)}_{\xi\eta} = \pd{r_{i,j}(\xi, \eta)}{(\xi, \eta)}
	= 
	\begin{bmatrix}
	0 & q_x + q_{xy}j & q_x + q_y + q_{xy}(i+j+1) & q_y + q_{xy} i
	\end{bmatrix}
	\pd{}{(\xi,\eta)}
	\begin{bmatrix}
	\PL_{-1, -1}(\xi, \eta)\\
	\PL_{1, -1}(\xi, \eta)\\
	\PL_{1, 1}(\xi, \eta)\\
	\PL_{-1, 1}(\xi, \eta)
	\end{bmatrix}.	
	\]
	The rightmost $4\times2$ matrix is independent of $i,j$. For shortness let us denote it as $\pd{\Phi}{(\xi, \eta)}$. Now
	\begin{multline*}
	J^{(i,j)}_{\xi\eta} = 
	\begin{bmatrix}
	0 & q_x + q_{xy}j & q_x + q_y + q_{xy}(i+j+1) & q_y + q_{xy} i
	\end{bmatrix}
	\pd{\Phi}{(\xi, \eta)} = \\ =
	\Big(
		\begin{bmatrix}
		0 & q_x & q_x + q_y + q_{xy} & q_y
		\end{bmatrix}
		+ 
		i
		\begin{bmatrix}
		0 & 0 & q_{xy} & q_{xy}
		\end{bmatrix}
		+ 
		j
		\begin{bmatrix}
		0 & q_{xy} & q_{xy} & 0
		\end{bmatrix}		
	\Big)
	\pd{\Phi}{(\xi, \eta)}.
	\end{multline*}
	The linearity of $J^{(i,j)}_{\xi\eta}$ in $i,j$ is now apparent. The proof for the $\left| J^{(i,j)}_{\xi\eta}\right|$ needs some more work. Let first rewrite $J^{(i,j)}_{\xi\eta}$ as:
	\[
	J^{(i,j)}_{\xi\eta} = J^{(0,0)}_{\xi\eta}
	+ q_{xy} \left(
	i \pd{(\Phi_{1,1} + \Phi_{-1,1})}{(\xi,\eta)}
	+ j \pd{(\Phi_{1,1} + \Phi_{1,-1})}{(\xi,\eta)}	
	\right) = 
	J^{(0,0)}_{\xi\eta} 
	+ \left[\frac{q_{xy}}{2}\right]^T
	\begin{bmatrix}
	i & j
	\end{bmatrix}.
	\]
	Using the matrix determinant lemma $\left|A+uv^T \right| = \left| A \right| + v^T \adj(A) u$:
	\[
	\left| J^{(i,j)}_{\xi\eta}\right| = \left| J^{(0,0)}_{\xi\eta}\right|
	+ 
	\frac{1}{2}
	\begin{bmatrix}
	i & j
	\end{bmatrix}
	\adj\left(J^{(i,j)}_{\xi\eta}\right)
	q_{xy}.
	\]
	So $\left| J^{(i,j)}_{\xi\eta}\right|$ is also a linear function in $i,j$. This finalizes the proof of the theorem.
\end{proof}

\subsection{Building a stiffness matrix and a force vector on a subdomain}\label{sec:stiffness_global}	
For calculating integrals value at present work we use the quadrature rule. We build a stiffness matrix on each subdomain by the following scheme:
\begin{enumerate}
	\item Map all finite elements on this subdomain to $\mathbb{K}=[-1, 1]^2$ using vectorized operations;
	
	\item\label{point:jacobians} Calculate Jacobian at the center of each finite element using the linear property of Jacobian from Section~\ref{sec:jacobian_linearity};
	
	\item Take one of possible combination of two corners in $\mathbb{K}$ (the corner $(-1, -1)$ and $(-1, 1)$, the corner $(-1, 1)$, and $(1, 1)$ and etc.). We denote these corners as $c_1=(i_1, j_1)$ and $c_2=(i_2, j_2)$;
	
	\item Calculate an approximation value of \eqref{eq:Aij} at the center of each finite element using values from point \ref{point:jacobians} and fact that $\nabla \Phi_{i_1 j_1}^L$ and $\nabla \Phi_{i_2 j_2}^L$ are the same for all mapping. After that, we obtain the value of the integral at all finite elements for a specific pair of shape function. Formula \eqref{eq:Aij} can be rewritten for the particular $F_e$ as:
	\begin{equation}\label{eq:Aij_reformulation}
	K_{c_1 c_2} = \iint_{\mathbb{K}} \JIT \nabla\PL_{i_1 j_1}\JIT\nabla\PL_{i_2 j_2}\left|\jac\right| d\xi d\eta\approx 4\nabla \PL_{i_1 j_1}\dfrac{\jac^A (\jac^A)^T}{|\jac|}\nabla\PL_{i_2 j_2},
	\end{equation}
	
	\begin{equation}
	\dfrac{J_{\xi\eta}^A \left(J_{\xi\eta}^A\right)^T}{|J_{\xi\eta}|} = 
	\dfrac{1}{J_{\xi\eta}}
	\begin{pmatrix}
	J_{22}^2 + J_{12}^ 2 & -J_{22}J_{21}-J_{12}J_{11}\\
	-J_{22}J_{21}-J_{12}J_{11} & J_{11}^2+J_{21}^2
	\end{pmatrix} = 
	\begin{pmatrix}
	J_{T_{11}} & J_{T_{12}}\\
	J_{T_{21}} & J_{T_{22}}
	\end{pmatrix},
	\end{equation}
	
	where $J_{kl}$ as a value from $k, l$ position of an adjugate matrix $J_{\xi\eta}^A$. Here we omit domain's indices $\xi, \eta$ from Jacobians for a formula simplification.
	
	We have used a rectangle rule for calculating this integral:
	\begin{equation}
	K_{c_1 c_2}= 4\left(J_{T_{11}}\dfrac{\partial\PL_{i_1 j_1}}{\partial\xi}\dfrac{\partial\PL_{i_2 j_2}}{\partial\xi} + 
	J_{T_{22}}\dfrac{\partial\PL_{i_1 j_1}}{\partial\eta}\dfrac{\partial\PL_{i_2 j_2}}{\partial\eta} + 
	J_{T_{12}}\left(\dfrac{\partial\PL_{i_1 j_1}}{\partial\eta}\dfrac{\partial\PL_{i_2 j_2}}{\partial\xi} + 
	\dfrac{\partial\PL_{ij}}{\partial\xi}\dfrac{\partial\PL_{i_2 j_2}}{\partial\eta}\right)\right).
	\end{equation}
	
	 It is important to mention that all operations in this scheme are vectorized and calculations are done for all finite elements simultaneously. Thus, $K_{c_1c_2}$ here is a vector, which contains integral values between $c_1$, $c_2$ for all finite elements in a mesh.
	 
	\item\label{point:Aij_shifting} A special block-matrix $V$ is used for shifting elements to their final position in a stiffness matrix by the following formula:
	\begin{equation}\label{eq:global_stiffness_matrix}
	A_{c_1 c_2} = V_{c_1}^T \mathrm{diag}\{K_{c_1 c_2}\} V_{c_2}.
	\end{equation}
	For details how $V$ is constructed refer to the Section \ref{sec:permutation_matrix};
	
	\item Repeat step \ref{point:jacobians}--\ref{point:Aij_shifting} for all pairs;
	\item To obtain the stiffness matrix on the domain, we have to sum all matrices together.
\end{enumerate}

Absolutely the same technique is used for the force vector:
\begin{equation}
	G_{c_1 c_2} = 4\left|J_{0}\right| \PL_{i_1 j_1}	\PL_{i_2 j_2},
\end{equation}
where $J_{0}$ is a Jacobian in the center of a finite element.

\begin{equation}\label{eq:global_force_vector}
	f_{c_1 c_2} = \left(V_{c_1}^T \mathrm{diag} \{G_{c_1 c_2}\} V_{c_2}\right)\fo.
\end{equation}

Now we are ready to construct the tensor representation of our matrix.

%% file: qtt.tex
\section{QTT-format}\label{sec:QTT_desc}

The \emph{Tensor Train decomposition} (TT) a non-linear, low-parametric tensors representation based on the separation of variables~\cite{oseledets2011tensor, oseledets2009breaking}. In present work we use a special case of TT -- \emph{Quantized Tensor Train decomposition} (QTT)~\cite{oseledets2011tensor, khor-qtt-2011, osel-2d2d-2010}  which packs simple vectors and matrices into a multidimensional tensor representation. This section explains basic ideas of TT and QTT.

Consider a $d$-dimensional tensor $T\in \mathbb{R}^{n_1\times \ldots \times n_d}$. The element with index $i_{1}\ldots i_{d}$ from $T$ is presented in TT-format as:
\begin{equation}\label{eq:tt-vector}
T_{i_{1}\ldots i_{d}}=\sum_{\alpha_1=1}^{r_1} \ldots \sum_{\alpha_{d-1}=1}^{r_{d-1}} G_1(i_1, \alpha_1) G_2(\alpha_1, i_2, \alpha_2)\ldots G_{d-1}(\alpha_{d-2}, i_{d-1}, \alpha_{d-1})G_d(\alpha_{d-1}, i_d), 
\end{equation}

where $0\leq i_k \leq n_k-1$, $k\in [1\ldots d]$. $G_j$ such that $j\in [2\ldots d-1]$ are three-dimension tensors with size $r_{j-1}\times n_j \times r_j$. $G_1$ and $G_d$ have the shapes $n_1 \times r_1$ and $r_{d-1} \times n_d$ respectively. The upper bounds $r_1,\ldots, r_{d-1}$ in sums \eqref{eq:tt-vector} are called the ranks of representation or ranks. A TT representation is not unique for each tensor. The lowest possible ranks among all possible representation of the tensor are called \emph{TT-ranks}~\cite{oseledets2011tensor}.

Formula \eqref{eq:tt-vector} can be rewritten as a product of matrices which depend on parameters:
\begin{equation}\label{eq:tt-vector-short}
T_{i_{1} \ldots i_{d}} = J_1(i_1)J_2(i_2)\dotsc J_{d-1}(i_{d-1}) J_d(i_d),
\end{equation}
where $J_k$ is a matrix with shape $r_{k-1}\times r_k$, $k\in[1\ldots d]$.

A memory amount which is neccessary for storing $T_{i_{1}\ldots i_{d}}$ can be calculated as:
\begin{equation}\label{eq:tt-memory-consumption}
	S = r_1 n_1 + \sum_{i=2}^{d-1} r_{i-1} n_i  r_i + r_{d-1} n_d.
\end{equation}

The \emph{effective rank (erank)} $r_e$ can be obtained from the solution of the following equation:
\begin{equation}
	S = r_1 n_1 + \sum_{i=2}^{d-1} r_{i-1} n_i  r_i + r_{d-1} n_d = r_e n_1 + \sum_{i=2}^{d-1} r_{e}^2 n_i + r_e n_d.
\end{equation}

The effective rank can be interpeted as an ``average'' rank for cores of tensor $T_{i_{1}\ldots i_{d}}$ and it is proportional to a square root from $S$.

Consider a multidimensional matrix $M\in \mathbb R^{(n_1 \times \ldots \times n_d) \times (m_1 \times \ldots \times m_d)}$. The element with indeces $i_1\ldots i_d;j_1\ldots j_d$ from $M$ is presented in TT as:
\begin{equation}\label{eq:tt-matrix}
M_{\substack{i_{1}\ldots i_{d}\\
		j_{1}\ldots j_{d}}}
=\sum_{\alpha_1=1}^{r_1} \ldots \sum_{\alpha_{d-1}=1}^{r_{d-1}} 
V_1(i_1, j_1, \alpha_1) 
V_2(\alpha_1, i_2, j_2, \alpha_2)
\ldots
V_{d-1}(\alpha_{d-2}, i_{d-1}, j_{d-1}, \alpha_{d-1})
V_d(\alpha_{d-1}, i_d, j_d), 
\end{equation}

where $0\leq i_k \leq n_k-1$, $0\leq j_k \leq m_k-1$, $V_j$ such that $j\in [2\ldots d-1]$ is a four-dimensional tensor with size $r_{j-1}\times n_j \times m_j \times r_{j}$, $V_1$ and $V_d$ have the shapes $n_1 \times m_1 \times r_{1}$ and $r_{d} \times n_d \times m_d$ respectively.

In the short form:
\begin{equation}\label{eq:tt-matrix-short}
M_{\substack{i_{1}\ldots i_{d}\\
		j_{1}\ldots j_{d}}}
=
Q_1(i_1, j_1)Q_2(i_2, j_2)\ldots Q_{d-1}(i_{d-1}, j_{d-1})Q_d(i_d, j_d), 
\end{equation}
where $Q_k(i_k,j_k)$ is a matrix with shape $r_{k-1}\times r_k$, $k\in[1\ldots d]$, pair $i_k, j_k$ is treated as one ``long index''. It is important to mention that cores $Q_k$ can be considered as a block matrix with a shape $r_{k-1}\times r_k$ and a block size $n_k\times m_k$.

The equation for \emph{erank} calculation for QTT-format is the following:
\begin{equation}
S = 2 r_1 + 2\sum_{i=2}^{d-1} r_{i-1}  r_i + 2 r_{d-1} = 4r_e+2r_e^2\left(d-2\right).
\end{equation}

The corollary of \eqref{eq:tt-matrix} is that if all TT-ranks of $M$ are equal to $1$, than $M$ may be present as a \emph{Kronecker product} of $d$ matrices:
\begin{equation}\label{eq:kron_prop}
M = M_1 \otimes M_2 \otimes \ldots \otimes M_{d-1}\otimes M_d,
\end{equation}
where $M_k$ is a matrix with the shape $n_k\times m_k$, $k\in[1\ldots d]$.

The storage cost and costs for basic operations in TT-format are bounded by $dnr^{\alpha}$, with $\alpha\in\{2, 3\}$ where $n\geq \max(n_1, \ldots n_d)$, $r\geq \max(r_1, \ldots r_{d-1})$. If $r$ is bounded in this case complexity is linear in $d$. For more information about properties of the TT-format refer to~\cite{oseledets2011tensor}.

QTT is used to apply TT to low-dimension objects like vectors and matrices. 
Consider a vector $\textbf{v}\in R^{2^d}$. This vector can be reshaped into a tensor $V$ with a shape $\underbrace{2\times 2 \times \ldots \times 2}_\text{d times}$. After that we can obtain TT representation of $V$ as a TT representation of a $d$-dimension tensor. The reshaping of $\textbf{v}$ to a $d$-dimensional tensor is equal to encoding its indeces $0\leq i\leq 2^d-1$ into a binary format:
\begin{equation}
i = \overline{i_1, i_2, \ldots i_d}=\sum_{k=1}^{d}2^{k-1}i_{k} \ \leftrightarrow \ (i_1, i_2, \ldots i_d),
\end{equation}
where $i_k\in \{0, 1\}$, $k\in[1\ldots d]$. 

The same idea is used to represent a matrix $\textbf{M}\in R^{2^d\times 2^d}$ in the QTT-format.

The last important thing we need to explain in this section is a special operation between two QTT cores $K$ and $L$ which are presented as block matrices. The ``bowtie'' operation between two block matrices is like a usual matrix product of two matrices, but their elements (blocks) being multiplied by means of the Kronecker product~\cite{kazeev2012low}. The operation is denoted by the symbol $\bowtie$.
\begin{align}\label{eq:bowtie-definition}
\begin{split}
C = K \bowtie L &= 
					\begin{bmatrix}
						K_{1 1} & \dotsc & K_{1 r_2}\\
						\vdots & \ddots & \vdots\\
						K_{r_1 1} & \dotsc & K_{r_1 r_2}
					\end{bmatrix}
					\bowtie
					\begin{bmatrix}
						L_{1 1} & \dotsc & L_{1 r_3}\\
						\vdots & \ddots & \vdots\\
						L_{r_2 1} & \dotsc & L_{r_2 r_3}
					\end{bmatrix} = \\
				&=
					\begin{bmatrix}
						K_{1 1} \otimes L_{1 1} + \dotsc + K_{1 r_2}\otimes L_{r_2 1} & \dotsc & K_{1 1} \otimes L_{1 r_3} + \dotsc + K_{1 r_2 }\otimes L_{r_2 r_3}\\
						\vdots & \ddots & \vdots\\
						K_{r_1 1} \otimes L_{11} + \dotsc + K_{r_1 r_2}\otimes L_{r_2 1} & \dotsc & K_{r_1 1} \otimes L_{1 r_3} + \dotsc + K_{r_1 r_2}\otimes L_{r_2 r_3}
					\end{bmatrix}.
\end{split}
\end{align}

In~\cite{kazeev2012low} was shown that a matrix $\textbf{M}\in R^{2^d\times 2^d}$ may be presented in the following form using the ``bowtie'' operation:
\begin{equation}\label{eq:QTT_bowtie}
	M = J_1\bowtie\dotsc\bowtie J_d,
\end{equation}

where $J_i$ is considered as a block matrix with size $r_{i-1}\times r_{i}$ and a block size $n_i \times m_i$, where $i \in [1, d]$.

\section{Z-ordering}
\subsection{Matrix ordering}\label{sec:z-ordering}

Consider a quadrilateral area $2^d\times 2^d$ with a quadrangle grid with $4^d$ nodes. Coordinates of some node in this grid are denoted by a pair of integers $\left(i, j\right)$, where $i, j \in \left[0, 2^{d}-1\right]$.

In this paper, we consider two types of nodes ordering, which are presented in Figure~\ref{img:ordering}:
\begin{enumerate}
	\item \emph{Canonical ordering}. A node with coordinates $(i, j)$ has its
		index calculated by the following formula:
		\begin{equation}\label{eq:canonical_ordering}
			\mathcal L\left(i, j\right) = i + 2^{d} j;
		\end{equation}
		
	\item \emph{Z-ordering}. Let us represent $\left(i, j\right)$ as:
		\begin{equation}\label{eq:z_ordering_power_pf_two}
					i = \sum_{k=1}^{d}2^{k-1}i_{k} \ \leftrightarrow \ (i_1, i_2, \ldots i_d),\ 
					j = \sum_{k=1}^{d}2^{k-1}j_{k} \ \leftrightarrow \ (j_1, j_2, \ldots j_d).
		\end{equation}
		The z-order index of a node at the point $\left(i, j\right)$ is calculated according to the following formula:
		\begin{equation}\label{eq:z_ordering_get_pos}
			\mathcal Z\left(i, j\right) = i_1 + 2 j_1 + 4 i_2 + 8 j_2 + \dotsc + 2^{2d-2} i_d + 2^{2d -1}j_d = \sum_{k=1}^{d} 2^{2k-2} i_k + \sum_{k=1}^{d} 2^{2k-1} j_k.
		\end{equation}		
\end{enumerate}

\begin{figure}[h]
	\centering
	\includegraphics[width=0.6\textwidth]{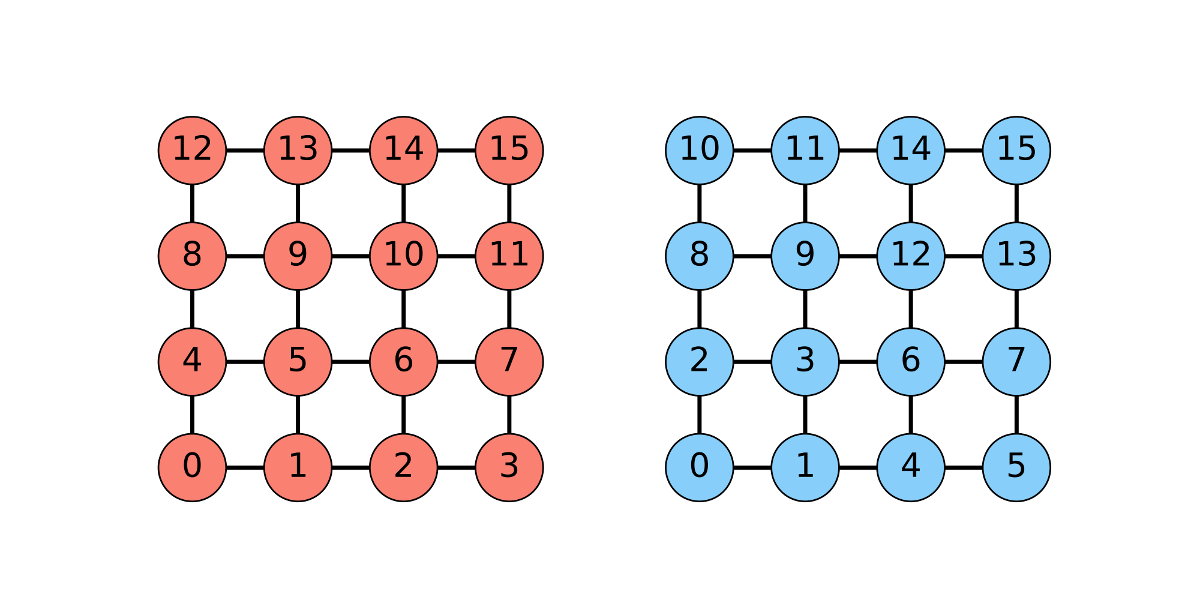}
	\caption{The ordering examples: left -- the canonical ordering, right -- the z-ordering.}
	\label{img:ordering}
\end{figure}

\subsection{Z-kron operation}\label{sec:z-kron}

Suppose we have two QTT-matrices:
\begin{equation}
	K_{\substack{i_1, \dotsc, i_d\\ j_1,\dotsc, j_d}} = K_1\left(i_1, j_1\right) \dotsc K_d\left(i_d, j_d\right),\ 
	L_{\substack{i'_1, \dotsc, i'_d\\ j'_1,\dotsc, j'_d}} = L_1\left(i'_1, j'_1\right) \dotsc L_d\left(i'_d, j'_d\right).
\end{equation}

Kronecker product of two $d$-dimensional QTT-matrices is a $2d$-dimensional QTT-matrix:
\begin{equation}
	M_{\substack{i_1,\dots, i_d, i'_1,\dotsc, i'_d\\ j_1,\dots, j_d, j'_1,\dotsc, j'_d}} =
	K_{\substack{i_1, \dotsc, i_d\\ j_1,\dotsc, j_d}}
	L_{\substack{i'_1, \dotsc, i'_d\\ j'_1,\dotsc, j'_d}} = 
	K_1\left(i_1, j_1\right) \dotsc K_d\left(i_d, j_d\right)
	L_1\left(i'_1, j'_1\right) \dotsc L_d\left(i'_d, j'_d\right).
\end{equation}

Let us define a new operation $\oslash$ -- z-kron operation, which applied to a pair of $d$-dimensional QTT-matrices yields a $d$-dimensional TT-matrix of size $\mathbb R^{4^d \times 4^d}$. The elements of this TT-matrix are defined in the similar way:
\begin{equation}
	M_{\substack{z_1, \dotsc, z_d\\ z'_1, \dotsc, z'_d}} = 
	K_{\substack{i_1, \dotsc, i_d\\ j_1,\dotsc, j_d}} L_{\substack{i'_1, \dotsc, i'_d\\ j'_1,\dotsc, j'_d}}\mbox{, where }z_k=i_k+2j_k,\ k\in[1\ldots d].
\end{equation}

\begin{theorem}\label{th:z-kron}
	The TT-tensor $M = K \oslash L$ can be expressed in terms of cores:
	\begin{equation}
		M_{\substack{z_1, \dotsc, z_d\\ z'_1, \dotsc, z'_d}}=
		M_1\left(z_1, z'_1\right)M_2\left(z_2, z'_2\right)\dotsc M_d\left(z_d, z'_d\right),
	\end{equation}
	where cores $M_k$ are obtained by regular matrix Kronecker product $M_k\left(z_k, z'_k \right)=K\left(i_k, j_k\right)\otimes L\left(i'_k, j'_k\right)$ and $z_k=i_k+2j_k$.
\end{theorem}

\begin{proof}
	Direct calculation gives
	\begin{align}
		\begin{split}
			M_1\left(z_1, z_1'\right)M_2\left(z_2, z_2'\right)\ldots M_d\left(z_d, z_d'\right) &= \left(K_1\left(i_1, j_1\right)\otimes L_1\left(i'_1, j'_1\right)\right)\ldots \left(K_d\left(i_d, j_d\right)\otimes L_d\left(i'_d, j'_d\right)\right)=\\
			&= \left(K_1\left(i_1, j_1\right)\ldots K_d\left(i_d, j_d\right)\right)\otimes\left(L_1\left(i'_1, j'_1\right)\ldots L_d\left(i'_d, j'_d\right)\right).
		\end{split}
	\end{align}
		
	The Kronecker product of two numbers is just a multiplication, that is why:
	\begin{align}
		\begin{split}
			\left(K_1\left(i_1, j_1\right)\ldots K_d\left(i_d, j_d\right)\right)\otimes\left(L_1\left(i'_1, j'_1\right)\ldots L_d\left(i'_d, j'_d\right)\right)
			&= K_1\left(i_1, j_1\right)\ldots K_d\left(i_d, j_d\right)L_1\left(i'_1, j'_1\right)\ldots L_d\left(i'_d, j'_d\right)\\
			&=
			M_{\substack{z_1, \dotsc, z_d\\ z'_1, \dotsc, z'_d}}.
		\end{split}
	\end{align}
\end{proof}

\subsection{Generation of a Z-order meshgrid in QTT-format}\label{sec:z-order-mesh-grid}
A \emph{meshgrid} is a standard tensor traversal generator which is avaible at  scientific libraries like numpy. Mesh grid can be counstructed using tensor operations. The mesh grid with the canonical ordering for a 2D matrix is constructed as:
\begin{equation}
	\io = L \otimes I,\ 	\jo = I \otimes L,
\end{equation}	

where $\io$ matrix indeces along the first dimension, $\jo$ matrix indeces along the second dimension,  $L$ is a vector, which contains a range from $0$ to $2^d-1$, $I$ is a vector with ones with a length $2^d$. 

According to Theorem \ref{th:z-kron} if we replace $\otimes$ to $\oslash$ we will obtain a mesh grid matrices with z-ordering:
\begin{equation}
	\iz = L \oslash I,\ 	\jz = I \oslash L.
\end{equation}

%% file: assembly.tex
\section{Matrix formulation of a concatenation}
\subsection{Concatenation matrices}\label{sc:concatination_matrix}

In Section~\ref{sec:solution_concatenation} the procedure of subdomain concatenation was expressed in terms of the \emph{connectivity matrix} $\Pi$. In this section we attempt to express $\Pi$ in a simple form that leads directly to an efficient TT representation.

If $m\neq p$, $\Pi_{mp}$ is defined as:
\begin{equation}
\left(\Pi_{mp}\right)_{\mathcal {IJ}}=\begin{cases}
1,\mbox{ if } \mathcal{I}^{(m)}\sim \mathcal{J}^{(p)}\\
0,\mbox{ otherwise}.
\end{cases}
\end{equation}
Otherwise:
\begin{equation}\label{eq:Pi_mm}
\Pi_{mm} = -\sum_{m\neq p} \Pi_{mp}\Pi_{pm}.
\end{equation}
Here $\mathcal I^{\left(m\right)} \sim \mathcal J^{\left(p\right)}$ is short for 
``node $\mathcal I$ from $\Omega^{(m)}$ should be concatenated with node $\mathcal J$ from $\Omega^{(p)}$''. $\Pi_{mm}$ is a diagonal matrix. The $i$'th number on the diagonal equals to the number of nodes $i$'th node is joined to.

The following possible cases of concatenation between two subdomains $m$ and $p$ exist:
\begin{enumerate}
	\item subdomains $m$ and $p$ do not share any nodes, $\Pi_{mp}=O$;
	\item subdomains $m$ and $p$ share one node (concatenated by vertex);
	\item subdomains $m$ and $p$ share nodes along a side (concatenated by a side), then they share $2^{d}$ nodes.
\end{enumerate}

For cases 2 and 3 there exist $16$ possible combinations for matrix $\Pi_{mp}$ since we can concatenate any side (vertex) of $\Omega^{(m)}$ with any side (vertex) of $\Omega^{(p)}$. The key idea of reducing the number of cases is to decompose the mapping $\Pi$ into a composition of three:
\begin{enumerate}
\item Map from $\Omega^{(m)}$ side (vertex) to some ``standard'' side (vertex).
\item Swap the standard side, that is map it onto itself in a reversed order. This step is not needed for the concatenation by vertex.
\item Map from the ``standard'' side (vertex) to $\Omega^{(p)}$ side (vertex).
\end{enumerate}

Concatenation by side is the most interesting operation. Let us focus on it.
According to the matrix $\Pi_{mp}$ can be represented as:
\begin{equation}\label{eq:Pi}
\Pi_{mp}=\left[\Psi_m\right]^T S \Psi_p,
\end{equation}
where $S$ is an exchange matrix a special case of a permutation matrix, where the 1 elements reside on the counterdiagonal and all other elements are zero. Now there are only four possible types of $\Psi$ matrices: one for each side of the quadrangle. Let them be $\Psi_B, \Psi_R, \Psi_T, \Psi_L$. The size of these matrices is $2^d \times 4^d$. Figure~\ref{img:concatenation_example} illustrates the concatenation of the bottom side of $\Omega^{(m)}$ with the left side of $\Omega^{(p)}$.

\begin{figure}[h]
	\centering
	\includegraphics[width=.5\textwidth]{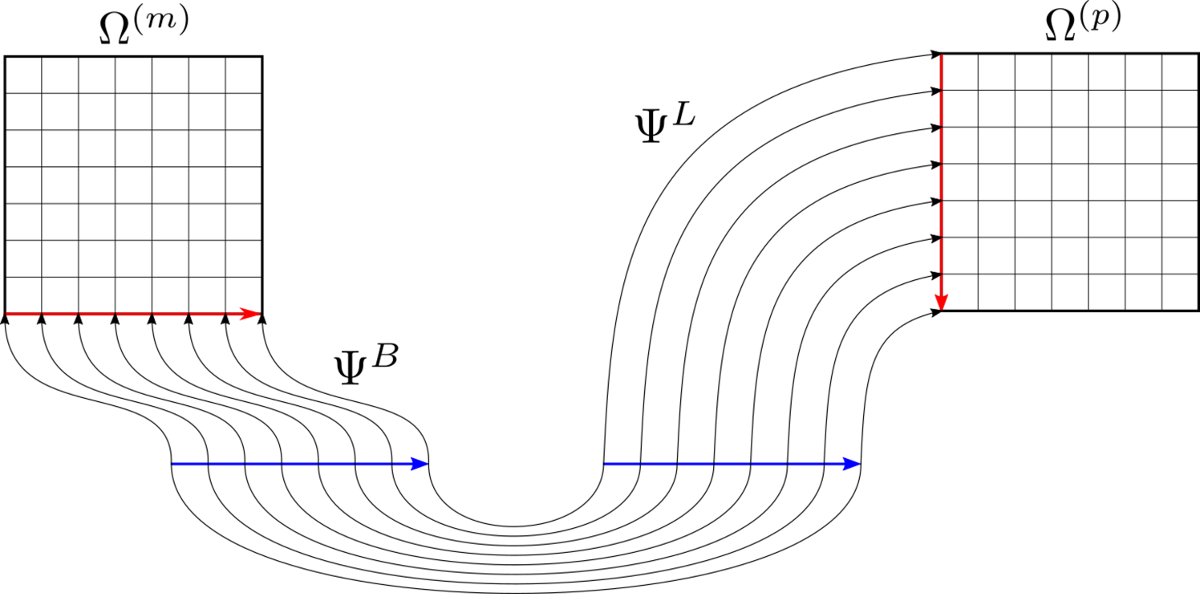}
	\caption{Mapping from the $\Omega^{(m)}$ bottom side to the $\Omega^{(p)}$ left side as a superposition of three mappings: transpose of $\Psi_B$, swap and finally $\Psi_L$.}
	\label{img:concatenation_example}
\end{figure}

Similarly for the vertex concatenation only 4 types of row vectors $\psi$ of size $1 \times 4^d$ are necessary. The matrix $\Pi_{mp}$ can be represented as
\[
\Pi_{mp} = \psi_m^\top \psi_p.
\]

\subsection{TT representation for a concatenation by side}\label{sec:concatination_by_side}

Let us now focus on the TT-representation for each of the four types of matrix $\Psi \in \mathbb R^{2^d \times 4^d}$. The first index $s$ of $\Psi$ enumerates nodes on the standard side, the second index $z$ of $\Psi$ enumerates subdomain nodes in a z-order. First, we quantize both of the indices:
\begin{gather*}
s = \sum_{k=1}^d 2^{k-1} s_k \leftrightarrow (s_1, s_2, \dots, s_d),\\
z = \sum_{k=1}^d 4^{k-1} z_k \leftrightarrow (z_1, z_2, \dots, z_d).
\end{gather*}
The matrix $\Psi$ is now reshaped into an $\underbrace{(2\times 2 \times \ldots \times 2)}_\text{d times}\times \underbrace{(4\times 4 \times \ldots \times 4)}_\text{d times}$ TT-matrix, where each element is indexed as:
\begin{equation}
\Psi_{\substack{s_1 s_2 \dotsc s_d\\ z_1 z_2 \dotsc z_d}},\ s_k\in\left\{0,1\right\},\ z_k \equiv i_k+2j_k,
\end{equation}
where $i_k, j_k$ are defined similar to the formula~\eqref{eq:z_ordering_power_pf_two}, $k\in[1\ldots 2^d]$.

According to~\eqref{eq:tt-matrix-short} each element of $\Psi$ can be presented in the TT-format in the following form:
\begin{equation}
	\Psi_{\substack{s_1 s_2 \dotsc s_d\\ z_1 z_2 \dotsc z_d}} = H_1\left(s_1, z_1\right)H_2\left(s_2, z_2\right)\dotsc H_d\left(s_d, z_d\right).
\end{equation}

In terms of the ``bowtie'' operation~\eqref{eq:bowtie-definition}-\eqref{eq:QTT_bowtie} looks like:

\begin{equation}\label{eq:Psi-bowtie}
	\Psi = H_1\bowtie \dotsc\bowtie H_d.
\end{equation}
Here $H_k$, $k = 1, \ldots, d$ are the cores of the TT decomposition of $\Psi$. 
It appears that each matrix $\Psi$ has TT-rank equal to $1$ and it cores can be explicitly given by
\begin{equation}
H_k = \begin{bmatrix}
\begin{pmatrix}
B & R & L & T \\
L & B & T & R
\end{pmatrix}
\end{bmatrix},
\end{equation} 
where $B, T, L, R$ are equal to 1 if the concatenation is performed on the bottom, the top, the left, or the right side, respectably, or, 0, otherwise.

For the defined enumerations the cores are as follows:
\begin{itemize}
	\item \emph{bottom side enumeration}:
	\begin{equation}\label{eq:bottom_core}
	H_k = \begin{bmatrix}
				\begin{pmatrix}
					1 & 0 & 0 & 0\\
					0 & 1 & 0 & 0
				\end{pmatrix}
			\end{bmatrix},\ k = 1, \ldots, d;
	\end{equation}	
	
	\item \emph{right side enumeration}:
	\begin{equation}\label{eq:right_core}
	H_k = \begin{bmatrix}
				\begin{pmatrix}
					0 & 1 & 0 & 0 \\
					0 & 0 & 0 & 1
				\end{pmatrix}
			\end{bmatrix},\ k = 1, \ldots, d;
	\end{equation}
	
	\item \emph{top side enumeration}:
	\begin{equation}\label{eq:top_core}
	H_k = \begin{bmatrix}
				\begin{pmatrix}
					0 & 0 & 0 & 1\\
					0 & 0 & 1 & 0
				\end{pmatrix}
			\end{bmatrix},\ k = 1, \ldots, d;
	\end{equation}
	
	\item \emph{left side enumeration}:
	\begin{equation}\label{eq:left_core}
	H_k = \begin{bmatrix}
				\begin{pmatrix}
					0 & 0 & 1 & 0 \\
					1 & 0 & 0 & 0
				\end{pmatrix}
			\end{bmatrix},\ k = 1, \ldots, d.
	\end{equation}
\end{itemize}

Note, to obtain $S\Psi$ from $\Psi$ it is sufficient to simply swap rows in each $H_k$.

\begin{theorem}
	\label{th:QTT-Psi}
	Matrix $\Psi$ can be represented as a TT matrix with identical cores equal to \eqref{eq:bottom_core},  \eqref{eq:right_core}, \eqref{eq:top_core}, \eqref{eq:left_core} for:
	
	\begin{itemize}
		\item \emph{bottom side enumeration}:
		\begin{equation}
		\Psi = \begin{bmatrix}
						\begin{pmatrix}
						1 & 0 & 0 & 0\\
						0 & 1 & 0 & 0
						\end{pmatrix}
					\end{bmatrix}^{\bowtie^d}
					=
						\begin{pmatrix}
						1 & 0 & 0 & 0\\
						0 & 1 & 0 & 0
						\end{pmatrix}^{\otimes^d};
		\end{equation}	
		
		\item \emph{right side enumeration}:
		\begin{equation}
		\Psi = \begin{bmatrix}
					\begin{pmatrix}
						0 & 1 & 0 & 0 \\
						0 & 0 & 0 & 1
					\end{pmatrix}
					\end{bmatrix}^{\bowtie^d}
					=
					\begin{pmatrix}
						0 & 1 & 0 & 0 \\
						0 & 0 & 0 & 1
					\end{pmatrix}^{\otimes^d};
		\end{equation}
		
		\item \emph{top side enumeration}:
		
		\begin{equation}
		\Psi = \begin{bmatrix}
					\begin{pmatrix}
					0 & 0 & 0 & 1\\
					0 & 0 & 1 & 0
					\end{pmatrix}
					\end{bmatrix}^{\bowtie^d}
					=
					\begin{pmatrix}
					0 & 0 & 0 & 1\\
					0 & 0 & 1 & 0
					\end{pmatrix}^{\otimes^d};
		\end{equation}
		
		\item \emph{left side enumeration}:
		\begin{equation}
		\Psi = \begin{bmatrix}
					\begin{pmatrix}
					0 & 0 & 1 & 0 \\
					1 & 0 & 0 & 0
					\end{pmatrix}
					\end{bmatrix}^{\bowtie^d}
					=
					\begin{pmatrix}
					0 & 0 & 1 & 0 \\
					1 & 0 & 0 & 0
					\end{pmatrix}^{\otimes^d}.
		\end{equation}
	\end{itemize}
	
\end{theorem}

\begin{proof}
Let us again recall the definition of $\Psi$:
\[
\Psi_{sz} = \begin{cases}
1, &\mbox{if } s \mbox { node of the standard side is mapped to node } z \mbox{ in 2D mesh},\\
0, &\mbox{otherwise}.
\end{cases}
\]
For each of the four types of $\Psi$ this expands into:
\begin{align}
(\Psi_B)_{sz} &= 	\begin{cases}
								1, &i = s, \; j = 0,\\
								0, &\mbox{otherwise}.
							\end{cases}\\
(\Psi_R)_{sz} &= 	\begin{cases}
								1, &i = 2^d - 1, \; j = s,\\
								0, &\mbox{otherwise}.
							\end{cases}\\
(\Psi_T)_{sz} &=  \begin{cases}
								1, &i = 2^d - 1 - s, \; j = 2^d - 1,\\
								0, &\mbox{otherwise}.
							\end{cases}\\
(\Psi_L)_{sz} &= 	\begin{cases}
								1, &i = 0, \; j = 2^d - 1 - s,\\
								0, &\mbox{otherwise}.
							\end{cases}
\end{align}

The conditions can be reformulated using $i_k, j_k$ and $s_k$ instead of $i, j, s$:
\begin{align}
i = s, \; j = 0 &\quad\Leftrightarrow\quad i_k = s_k, \; j_k = 0,\\
i = 2^d - 1, \; j = s &\quad\Leftrightarrow\quad i_k = 1, \; j_k = s_k,\\
i = 2^d - 1 - s, \; j = 2^d - 1 &\quad\Leftrightarrow\quad i_k = 1 - s_k, \; j_k = 1,\\
i = 0, \; j = 2^d - 1 - s &\quad\Leftrightarrow\quad i_k = 0, \; j_k = 1 - s_k.
\end{align}
Finally, the conditions can be simplified further using $z_k = i_k + 2 j_k$ instead of $i_k, j_k$:
\begin{itemize}
\item for $\Psi_B$: $z_k = s_k$;
\item for $\Psi_R$: $z_k = 1 + 2s_k$;
\item for $\Psi_T$: $z_k = 3 - s_k$;
\item for $\Psi_L$: $z_k = 2 - 2s_k$.
\end{itemize}
We can see that after the quantization the conditions separate for different $k$. This separation allows us to write $\Psi$ as a Kronecker product of $d$ matrices of size $2 \times 4$ with nonzeros given by above rules:
\[
\Psi = \begin{pmatrix}
	B & R & L & T \\
	L & B & T & R
	\end{pmatrix}^{\otimes^d}.
\]
The corresponding TT representation follows immediately.
\end{proof}

\subsection{TT representation for a concatenation by vertex}\label{sec:concatination_by_vertex}
If two subdomains share only one corner point, $\Psi$ can be derived from the fact that the first row of \eqref{eq:bottom_core}-\eqref{eq:left_core} defines the quad corners.
Consider the core $G_k$ which is derived from $H_k$ by taking the first row:
\begin{equation}
G_k = \begin{pmatrix}
			\begin{bmatrix}
				B & R & L & T
			\end{bmatrix}
		\end{pmatrix}.
\end{equation}

$B$ describes the left bottom corner ($LB$), $R$ describes the right bottom corner ($RB$), $L$ describes the left top corner ($LT$), $T$ describes the right top corner ($RT$). Using this notation, we get:
\begin{equation}
G_k = \begin{pmatrix}
			\begin{bmatrix}
				LB & RB & LT & RT
			\end{bmatrix}
		\end{pmatrix}.
\end{equation}

\subsection{Building the permutation matrix for a subdomain}\label{sec:permutation_matrix}
\newenvironment{comment}{}{}

In Section \ref{sec:stiffness_global} we mentioned the matrix $V_{c}$ a \emph{shift matrix} for corner coordinate $c=(x,y)$ in a finite element. Each matrix $V_{c}$ is built from two matrices $W_{0}$ and $W_{1}$. These matrices are shift matrices for a one-dimensional finite element. $W_{0}$ is an $2^d\times 2^d$ matrix defined as:	
\begin{equation}
W_0(e, p) =	\begin{cases}
						1, e=p, e\neq 2^d-1,\\
						0,\mbox{ otherwise}
					\end{cases}
\end{equation}

where $e$ is a number of a finite element in the subdomain, $p$ is a number of node.

$W_{1}$ is an $2^d\times 2^d$ matrix defined as:
\begin{equation}
W_1(e, p) =	\begin{cases}
						1, p=e+1,\\
						0,\mbox{ otherwise}.
					\end{cases}
\end{equation}

It is easy to see that both matrices $W$ with $2^d$ nodes can be written as:
\begin{align}
\begin{split}
W^{(d)} &= M_1 \otimes R_2 \otimes R_3 \otimes \dotsc \otimes R_{d-1} \otimes R_d + \\
&+ L_1 \otimes M_2 \otimes R_3 \otimes \dotsc \otimes R_{d-1} \otimes R_d +\\
&+ \dotsc + L_1 \otimes L_2 \otimes L_3 \otimes \dotsc \otimes 	L_{d-2}\otimes M_{d-1}\otimes R_d + \\
&+ L_1 \otimes L_2 \otimes \dotsc \otimes L_{d-1} \otimes M_d,
\end{split}
\end{align}

where $M_i$, $R_i$, $L_i$ for $W_0$ are equal to:
\begin{equation}
M_i = \begin{pmatrix}
1 & 0 \\
0 & 0
\end{pmatrix},\  
R_i = \begin{pmatrix}
1 & 0 \\
0 & 1
\end{pmatrix},\ 
L_i = \begin{pmatrix}
0 & 0\\
0 & 1
\end{pmatrix}.
\end{equation} 

For $W_1$:
\begin{equation}
M_i = \begin{pmatrix}
0 & 0 \\
1 & 0
\end{pmatrix},\  
R_i = \begin{pmatrix}
0 & 1 \\
0 & 0
\end{pmatrix},\ 
L_i = \begin{pmatrix}
1 & 0\\
0 & 1
\end{pmatrix}.
\end{equation}

Hence, according to the ~\cite{kazeev2012low} $W_0$ in the QTT-format for $d\geq 2$ is presented as a composition of ``bowtie'' operations between block matrices:
\begin{equation}\label{eq:W_0_QTT}
W_0 = 	\begin{bmatrix}
\begin{pmatrix}
0 & 0 \\
0 & 1
\end{pmatrix} 
&
\begin{pmatrix}
1 & 0\\
0 & 0
\end{pmatrix}
\end{bmatrix}
\bowtie
\begin{bmatrix}
\begin{pmatrix}
0 & 0 \\
0 & 1
\end{pmatrix} 
&
\begin{pmatrix}
1 & 0\\
0 & 0
\end{pmatrix} \\
\begin{pmatrix}
0 & 0 \\
0 & 0
\end{pmatrix} 
&
\begin{pmatrix}
1 & 0\\
0 & 1
\end{pmatrix} 
\end{bmatrix}^{\bowtie^{d-2}}
\bowtie
\begin{bmatrix}
\begin{pmatrix}
1 & 0\\
0 & 0
\end{pmatrix}\\
\begin{pmatrix}
1 & 0\\
0 & 1
\end{pmatrix} 					
\end{bmatrix}.
\end{equation}

The matrix $W_1$ in the QTT-format for $d\geq 2$ has the following form:
\begin{equation}
W_1 = 	\begin{bmatrix}
\begin{pmatrix}
1 & 0 \\
0 & 1
\end{pmatrix} 
&
\begin{pmatrix}
0 & 0\\
1 & 0
\end{pmatrix}
\end{bmatrix}
\bowtie
\begin{bmatrix}
\begin{pmatrix}
1 & 0 \\
0 & 1
\end{pmatrix} 
&
\begin{pmatrix}
0 & 0\\
1 & 0
\end{pmatrix} \\
\begin{pmatrix}
0 & 0 \\
0 & 0
\end{pmatrix} 
&
\begin{pmatrix}
0 & 1\\
0 & 0
\end{pmatrix} 
\end{bmatrix}^{\bowtie^{d-2}}
\bowtie
\begin{bmatrix}
\begin{pmatrix}
0 & 0\\
1 & 0
\end{pmatrix}\\
\begin{pmatrix}
0 & 1\\
0 & 0
\end{pmatrix} 					
\end{bmatrix}.
\end{equation}

In the edge-case of $d=1$ $W_0$ and $W_1$:
\begin{equation}\label{eq:W_edge-case}
W_0 = 	\begin{bmatrix}
\begin{pmatrix}
1 & 0\\
0 & 0
\end{pmatrix}
\end{bmatrix}
\ 
W_1 = 	\begin{bmatrix}
\begin{pmatrix}
0 & 1\\
0 & 0
\end{pmatrix}
\end{bmatrix}.
\end{equation}

From $W_0$ and $W_1$ we can construct 4 different $V_{c}$ for different corners in $\mathbb{K}$:
\begin{equation}
	\begin{array}{|r|r|c|}
		\hline
		x & y & V_c\\
		\hline
		-1 & -1 & W_0\otimes W_0\\
		-1 & 1 & W_0\otimes W_1\\
		1 & -1 & W_1\otimes W_0\\
		1 & 1 & W_1\otimes W_1\label{eq:V_11}\\
		\hline
	\end{array}
\end{equation}

According to Theorem \ref{th:z-kron} if we replace the Kronecker product $\otimes$ in \eqref{eq:V_11} to the z-kron product $\oslash$ we will get a shift matrix in z-order.

\section{Boundary conditions in subdomains}\label{sec:condition_on_subdomain}
After building all matrices from \eqref{eq:fem_concatinated} with the help of \eqref{eq:stiffness_interfaces}, \eqref{eq:Pi_mm}, \eqref{eq:Pi} and \eqref{eq:global_stiffness_matrix} we need to apply the Dirichlet condition to the external boundaries.

Let us define a vector $X$, which represent a boundary mask for a one-dimensional interval with $2^d$ points on it. There are 4 possible configurations of this vector. In the case of the Dirichlet condition from the beginning of the vector, we denote as $X_{DN}$ and this vector contains ones everywhere, except the first element. $X_{ND}$ for the Dirichlet at the end of interval contains ones everywhere, except the last element in a vector. For the Dirichlet condition at beginning and at the end we introduce $X_{DD}$, it contains ones everywhere, except the first and the last element. And without the Dirichlet condition it is $X_{NN}$, it contains ones only.

To construct a \emph{boundary mask} for a two-dimensional case we should take a Kronecker product of two possible configurations of a vector $X$. Here are a few examples for left, right, bottom and top sides:
\begin{equation}
	\begin{array}{c}
		m_{L} = X_{DN}\otimes X_{NN},\ m_{R} = X_{ND}\otimes X_{NN},\\
		m_{B} = X_{NN}\otimes X_{DN},\ m_{T} = X_{NN}\otimes X_{ND}.\label{eq:m_examples}
	\end{array}
\end{equation}

For example, $m_{L}$ has the following structure:
\begin{equation}
m_{L} = \begin{pmatrix}
				0 & 1 & \dotsc & 1 & \dotsc & 1\\
				\vdots & \vdots & \ddots & \vdots & \ddots & \vdots \\
				0 & 1 & \dotsc & 1 & \dotsc & 1
			\end{pmatrix}.
\end{equation}

This matrix is a map, to what node the Dirichlet condition is applied and to what is not. Again, according to Theorem \ref{th:z-kron} by the replacement $\otimes$ to $\oslash$ in ~\eqref{eq:m_examples} we will get $m$ in the z-order.

A mask $m$ is applied to the stiffness matrix $A$ and the force vector $f$ of each subdomain by the following formulas:
\begin{equation}\label{eq:stiffness_applied}
	A_{\text{applied}} = \mathrm{diag}\left(m\right)A+(E-\mathrm{diag}\left(m\right)),
\end{equation}

\begin{equation}\label{eq:force_applied}
	f_{\text{applied}} = m\circ f,
\end{equation}

the first summand of Equation~\eqref{eq:stiffness_applied} fills zero rows, which corresponds to the nodes on the boundary with the Dirichlet condition, and the second one sets ones to diagonal for convenience. $\circ$ is an \emph{element-wise product}.

\section{Construction of a final stiffness matrix and a force vector}\label{sec:build_final_matrix}

We work under the assumption that the number of subdomains is much less than the number of nodes in meshes. The \emph{final stiffness matrix} and the \emph{final force vector} from \eqref{eq:fem_concatinated} are built using the Kronecker product. For example from Figure~\ref{img:triangle} we have the following matrices: $B_{11}$, $B_{12}$, $B_{13}$, $B_{22}$, $B_{21}$, $B_{23}$, $B_{31}$, $B_{32}$, $B_{33}$. We build the final stiffness matrix in the following manner:
\begin{align}
\begin{split}
A &= \begin{pmatrix}
1 & 0 & 0\\
0 & 0 & 0\\
0 & 0 & 0
\end{pmatrix}
\otimes
B_{11}
+
\begin{pmatrix}
0 & 1 & 0\\
0 & 0 & 0\\
0 & 0 & 0
\end{pmatrix}
\otimes
B_{12}
+
\begin{pmatrix}
0 & 0 & 1\\
0 & 0 & 0\\
0 & 0 & 0
\end{pmatrix}
\otimes
B_{13} +\\
&+
\begin{pmatrix}
0 & 0 & 0\\
1 & 0 & 0\\
0 & 0 & 0
\end{pmatrix}
\otimes
B_{21}
+ 
\begin{pmatrix}
0 & 0 & 0\\
0 & 1 & 0\\
0 & 0 & 0
\end{pmatrix}
\otimes
B_{22} 
+ 
\begin{pmatrix}
0 & 0 & 0\\
0 & 0 & 1\\
0 & 0 & 0
\end{pmatrix}
\otimes
B_{23} +\\
&+ 
\begin{pmatrix}
0 & 0 & 0\\
0 & 0 & 0\\
1 & 0 & 0
\end{pmatrix}
\otimes
B_{31}
+ 
\begin{pmatrix}
0 & 0 & 0\\
0 & 0 & 0\\
0 & 1 & 0
\end{pmatrix}
\otimes
B_{32}
+ 
\begin{pmatrix}
0 & 0 & 0\\
0 & 0 & 0\\
0 & 0 & 1
\end{pmatrix}
\otimes
B_{33}.
\end{split}
\end{align}

And the force vector:
\begin{align}
\begin{split}
g = \begin{pmatrix}
1\\
0\\
0
\end{pmatrix}
\otimes
g^{(1)}
+
\begin{pmatrix}
0\\
1\\
0
\end{pmatrix}
\otimes
g^{(2)}
+
\begin{pmatrix}
0\\
0\\
1
\end{pmatrix}
\otimes
g^{(3)}.
\end{split}
\end{align}

%% file: experiments.tex
\section{Numerical experiments}

\begin{figure}[h]
	\centering
	\includegraphics[width=\textwidth]{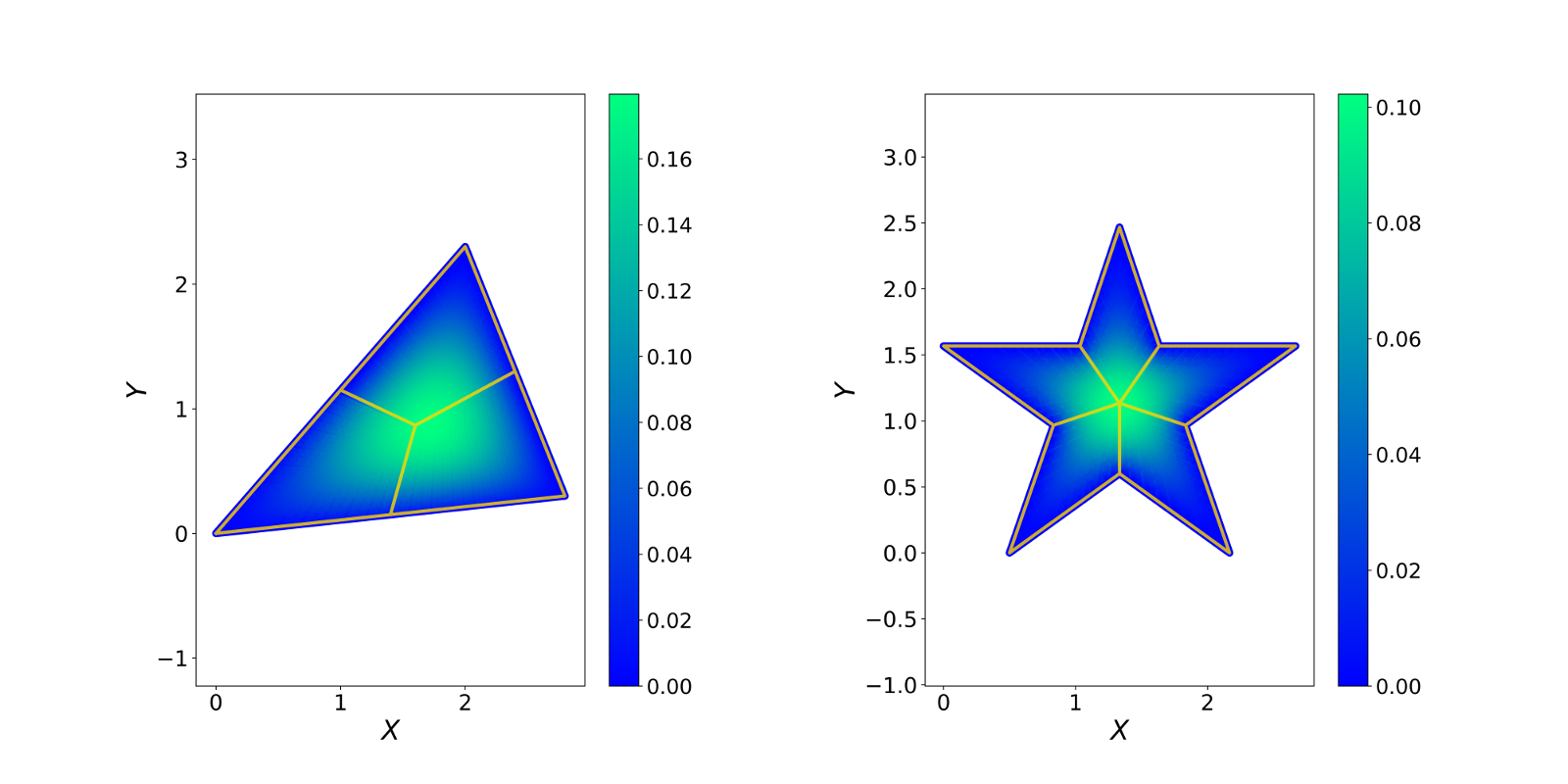}
	\caption{Experemental domains: the triangle (left) and the star (right). Yellow contour shows domain decomposition into quadrangle subdomains.}
	\label{img:subdomains}
\end{figure}

In this section a numerical expirement is provided -- we numerically solve a 2D dimensional Dirichlet problem. We chose two shapes for $\Omega$ -- a triangular shape and a star shape, as shown in Figure~\ref{img:subdomains}. Our goal is to study how numerical solutions behave, how eranks grow and to identify peak memory consumption with different approximation accuracies and mesh resolutions. We compare our approach which is described by Algorithm \ref{alg:final_alg} with FEniCS solution\cite{logg2012automated} on the same domain. To obtain the solution in QTT-format we use the AMEN solver \cite{dolgov2014alternating}. Memory consumption measurement was done by the memory profiler package for python 2.7.

Each numerical experiment consists of the following steps:

\begin{enumerate}
	\item Find a numerical solution of \eqref{eq:Dirichlet_problem} with various mesh sizes using FEniCS~\cite{logg2012automated}.
	\item Divide the domain into quadrangle subdomains, according to Figure~\ref{img:subdomains}.
	\item Find a numerical solution of \eqref{eq:Dirichlet_problem} using the Algorithm \ref{alg:final_alg} with some approximation accuracy $\epsilon$.\label{itm:find_solution}
	\item Find the energy $E_s$ of the solution from the previous step.\label{itm:find_energy}
	\item Repeat \ref{itm:find_solution}-\ref{itm:find_energy} with various $\epsilon$ and vertex counts.
	\item Find the solution energy approximation $E$ using Richardson extrapolation~\cite{liu2006fundamental} with energies from step~\ref{itm:find_energy}.\label{itm:aprox_solution}
\end{enumerate}

Figures~\ref{img:bench_tri},~\ref{img:bench_star},~\ref{img:eranks} present the results of our experiments.

The left plots in Figures~\ref{img:bench_tri},~\ref{img:bench_star} show an energy error -- the difference between the approximated solution energy $E$ from step~\ref{itm:aprox_solution} and the energies for solutions given by Algorithm~\ref{alg:final_alg} and FEniCS with various approximation accuracies and a vertex counts. The Algorithm~\ref{alg:final_alg} demonstrates the second-order convergence. The Algorithm~\ref{alg:final_alg} has smaller energy error than FEniCS solution until a grid becomes too fine. After that the approximation error becomes too small for the working precision. The middle plots in Figures~\ref{img:bench_tri},~\ref{img:bench_star} show exponential growths of erank w.r.t. approximation accuracy. The right plots demonstrate that our approach significantly improved the peak memory consumption in comparison with FEniCS.

\begin{figure}[h]
	\centering
	\includegraphics[width=\textwidth]{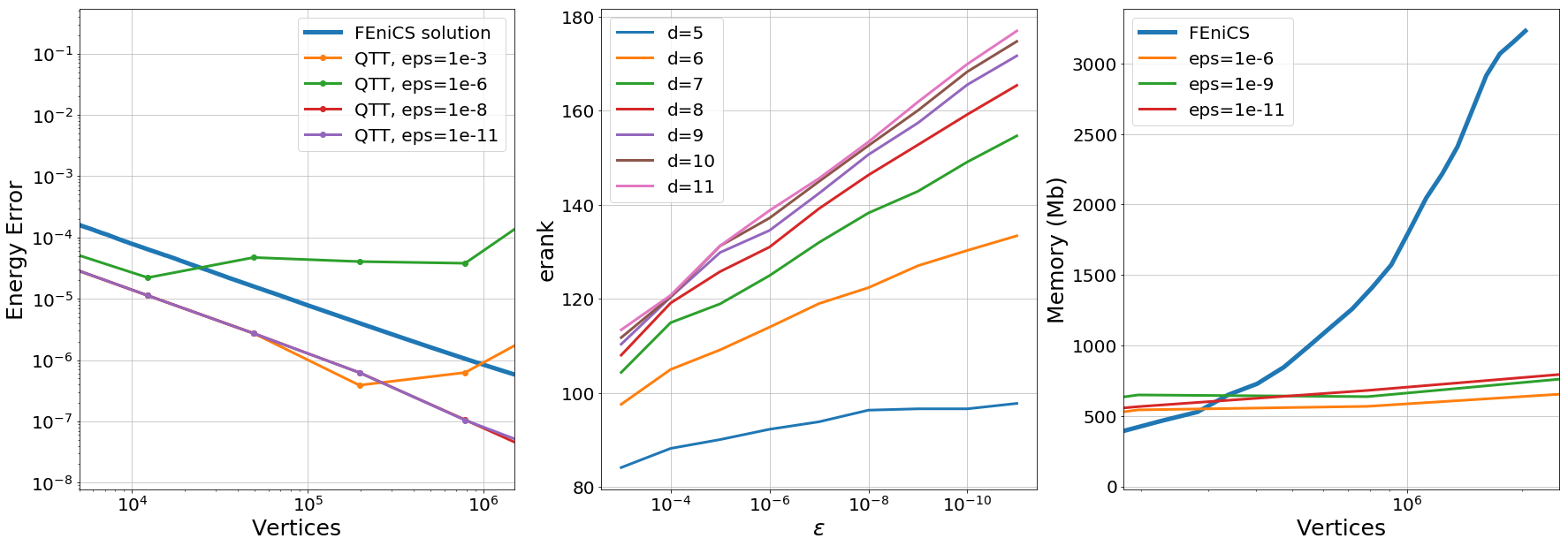}
	\caption[x]{Benchmark for the triangle domain. The left plot is the energy error from vertices number. The middle plot shows erank growth w.r.t. approximation accuracy. The right plot is the peak memory consumption w.r.t. vertex counts.}
	\label{img:bench_tri}
\end{figure}

\begin{figure}[h]
	\centering
	\includegraphics[width=\textwidth]{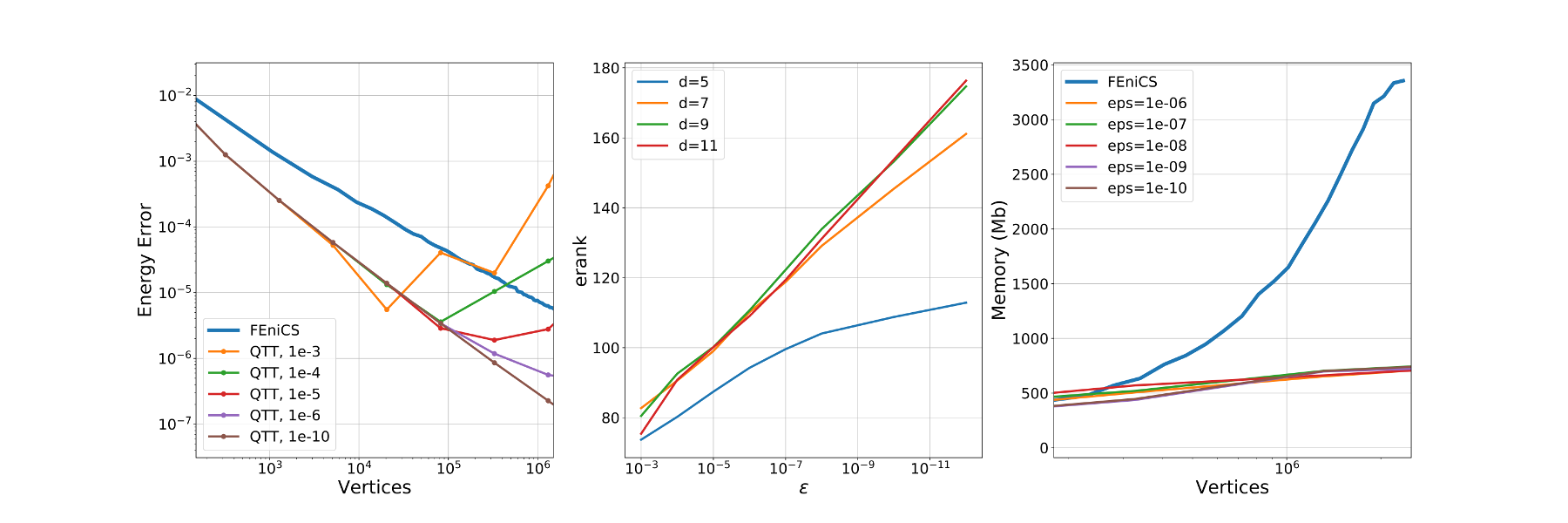}
	\caption[x]{Benchmark for the star domain. The left plot is energy error from vertices number. The middle plot is erank growing from approximation accuracy. The right plot is a peak memory consumption from vertices number.}
	\label{img:bench_star}
\end{figure}

Figure~\ref{img:eranks} presents the growth of erank for stiffness matrices and force vectors for both domains, and shows how energy converges for FEniCS and the Algorithm~\ref{alg:final_alg} on the star domain. The middle plot presents erank growth for matrices of which matrix $B$ of~\eqref{eq:fem_concatinated} consists of. $B_{ii}$ are submatrices located at the diagonal of $B$, and $B_{ij}$ are submatrices at off-diagonal positions of $B$. $A_i$ are subdomain stiffness matrices before the concatenation conditions applied. All matrices shown on this plot demonstrate slow erank growth w.r.t. the number of vertices. The middle plot presents the dynamics of erank changing for the whole matrix $B$, the final force vector $g$, and its parts for each subdomain $g^{(m)}$ from~\eqref{eq:fem_concatinated}. $g$ and $g^{(m)}$ have the same behavior for star and triangle domains. $B$, $g$ and $g^{(m)}$ all demonstrate a slow erank growth w.r.t. the number of vertices. The right plot presents energy convergence for FEniCS and the Algorithm~\ref{alg:final_alg}. It is important to mention that FEniCS energy approaches the approximated energy from the bottom whereas the Algorithm~\ref{alg:final_alg} energy approaches from the top. This means that FEniCS gives a lower bound of the solution energy and the Algorithm~\ref{alg:final_alg} gives an upper bound of the solution energy.

\begin{figure}[h]
	\centering
	\includegraphics[width=0.66\textwidth]{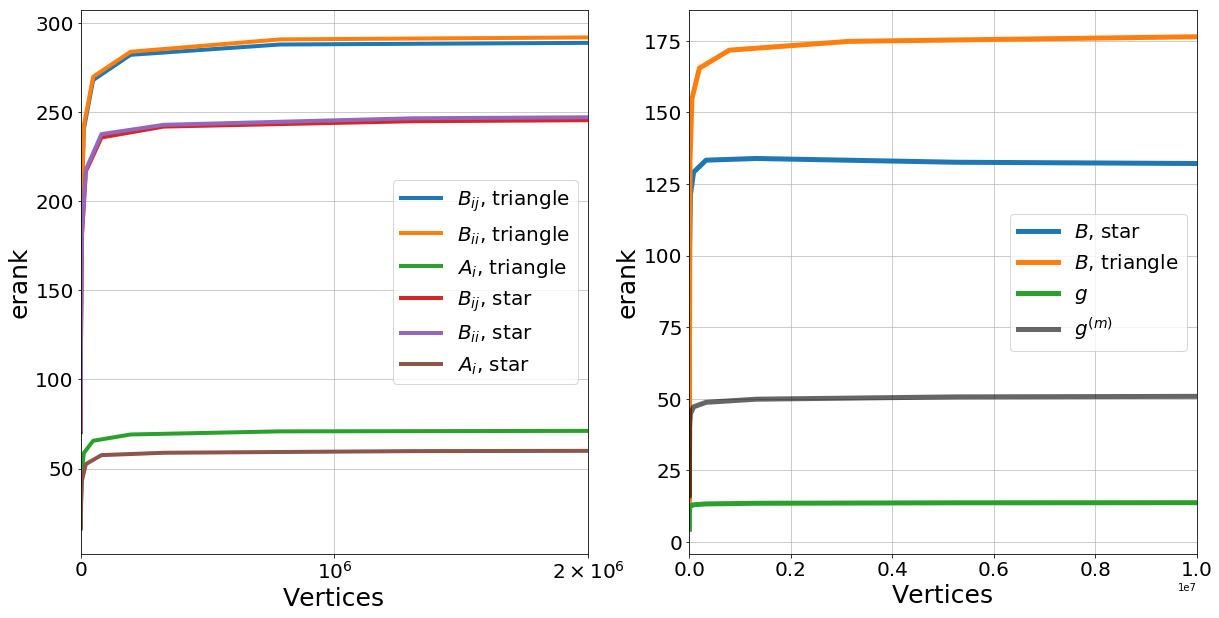}
	\caption[x]{The left plot is eranks for parts of matrix $B$ from~\eqref{eq:fem_concatinated}. The middle plot is eranks for matrix $B$, vector $g$ and vectors $g^{(m)}$ from~\eqref{eq:fem_concatinated}. The right plot shows solution convergence for FEniCS and the Algorithm~\ref{alg:final_alg} for the star-shaped domain.}
	\label{img:eranks}
\end{figure}

\section{Conclusion}
In this paper, we proposed a numerical algorithm that solves a two-dimensional elliptic problem in a polygonal domain using QTT and isogeometric analysis. We proposed the discretisation scheme that can be used to build the final stiffness matrix in the QTT-format with logarithmic complexity. This scheme involves a solution concatenation between sub-domains.

This scheme allows to use successfully QTT-format advantages and avoids some problems like approximation ranks growth. The approximation ranks growth was avoided by introduction a special matrix nodes ordering -- z-order, and a new operation -- Z-kron, which helps to construct the ordering. The present paper presents the way of building a mesh grid for 2D dimension case in z-order using this operation. We show how to build auxiliary matrices in z-order for solution concatenation.

We have shown that it is possible to build the final stiffness matrix in QTT-format with help of Z-kron ``on the fly'' as opposed to the transformation of a calculated matrix into QTT. This fact allows us to decrease approximation ranks and peak memory consumption.

Finally, our experiments show that our algorithm has the second-order convergence, consumes less memory than FEniCS and gives an upper bound estimation for the solution.